\documentclass[11pt]{article}
\usepackage{amsmath,amsthm,amsfonts,amssymb,bbm}
\usepackage[sort, numbers, square]{natbib}
\usepackage{hyperref}

\usepackage{enumerate}
\usepackage{multirow} 
\usepackage{graphicx}
\usepackage[all]{xy}
\usepackage{tabularx} 
\usepackage{booktabs} 
\usepackage{accents} 
\usepackage[font=small,labelfont=bf]{caption} 

\usepackage{color} 

\newcommand{\cf}{cf.\ } 
 
\newcommand{\page}{p.\ } 
\newcommand{\ppage}{pp.\ } 
\newcommand{\resp}{resp.\ }

\newcommand{\ie}{i.e.\ } 
\newcommand{\eg}{e.g.\ } 
\newcommand{\iid}{i.i.d.\ } 
 
\newcommand{\wrt}{w.r.t.\ }

\newcommand{\RR}{\ensuremath{\mathbb{R}}}

\newcommand{\NN}{\ensuremath{\mathbb{N}}}

\newcommand{\PP}{\ensuremath{\mathbb{P}}}
\newcommand{\EE}{\ensuremath{\mathbb{E}}}
\newcommand{\FF}{\ensuremath{\mathbb{F}}}

\DeclareMathOperator{\Var}{Var}

\newcommand{\trans}{\ensuremath{\text{\normalfont T}}}
\newcommand{\D}{\ensuremath{\text{\normalfont d}}}

\newcommand{\erf}{\ensuremath{\text{\normalfont erf}}}
\newcommand{\erfc}{\ensuremath{\text{\normalfont erfc}}}

\newcommand{\BR}{\ensuremath{\text{\normalfont BR}}}
\newcommand{\VBR}{\ensuremath{\text{\normalfont VBR}}}
\newcommand{\EG}{\ensuremath{\text{\normalfont EG}}}
\newcommand{\EBG}{\ensuremath{\text{\normalfont EBG}}}
\newcommand{\MMM}{\ensuremath{\text{\normalfont M3}}}
\newcommand{\MMMr}{\ensuremath{\text{\normalfont M3r}}}

\newcommand{\MMr}{\ensuremath{\text{\normalfont M2r}}}
\newcommand{\MMMb}{\ensuremath{\text{\normalfont M3b}}}

\newcommand{\MPS}{\ensuremath{\text{\normalfont MPS}}}
\newcommand{\tb}{\ensuremath{\text{\normalfont TB}}}

\newcommand{\orth}{\ensuremath{\text{\normalfont O}}}

\newcommand{\Borel}{\ensuremath{{\mathcal B}}}

\newcommand{\vol}{\ensuremath{\nu}}
\newcommand{\eins}{\ensuremath{\mathbf{1}}}
\newcommand{\Eins}{\ensuremath{\mathbbm{1}}}
\newcommand{\distr}{\ensuremath{\mathcal{D}}}
\newcommand{\laplacetransform}{\ensuremath{\mathcal{L}}}

\newcommand{\pmid}{\ensuremath{\,:\,}}
\newcommand{\interior}[1]{\accentset{\circ}{#1}}
\newcommand{\Max}{\ensuremath{\bigvee}}

\newcommand{\pintensity}{\beta} 

\theoremstyle{plain}
\newtheorem{theorem}{Theorem}
\newtheorem{proposition}[theorem]{Proposition}

\newtheorem{lemma}[theorem]{Lemma}

\theoremstyle{definition}

\newtheorem{example}[theorem]{Example}

\theoremstyle{remark}
\newtheorem{remark}[theorem]{Remark}

\begin{document}

\title{Systematic co-occurrence of tail correlation functions \\ among max-stable processes} 

\author{
  Kirstin Strokorb\footnote{Institute of Mathematics, University of Mannheim, D-68131 Mannheim, Email: strokorb@math.uni-mannheim.de},
  Felix Ballani\footnote{Institute of Stochastics, Faculty of Mathematics and Computer Science, TU Bergakademie Freiberg, D-09596 Freiberg, Email: ballani@math.tu-freiberg.de},
  Martin Schlather\footnote{Institute of Mathematics, University of Mannheim, D-68131 Mannheim, Email: schlather@math.uni-mannheim.de}
}

\maketitle 

\begin{abstract}
The tail correlation function (TCF) is one of the most popular bivariate extremal dependence measures that has entered the literature under various names. We study to what extent the TCF can distinguish between different classes of well-known max-stable processes and identify essentially different processes sharing the same TCF. 
\end{abstract}

{\small
\phantom{\textit{2010 MSC}:} {Secondary 60G60}
}


\section{Introduction}\label{sec:intro}

The \emph{tail correlation function (TCF)} $\chi$ of a stationary
process $X$ on $\RR^d$ is defined through 
\begin{align*}
\chi(t):=\lim_{\tau \uparrow \tau_0} \PP(X_t \geq \tau \mid X_o \geq \tau), \qquad t \in \RR^d,
\end{align*}
provided the limit exists.
Here, $\tau_0$ is the upper endpoint of the univariate marginal distribution and $\chi$ does not depend on the choice of one-dimensional marginals. 
Dating back to \cite{geffroy_5859,sibuya_60,tiagodeoliveira_62} the TCF is one of the most popular bivariate extremal dependence measures that has entered the literature under various names, most prominently \emph{(upper) tail dependence coefficient} \cite{beirlantetalii_03,davismikosch_09,falk_05}, \emph{$\chi$-measure} \cite{beirlantetalii_03,colesheffernantawn_99} or \emph{extremal coefficient function} \cite{fasenetalii_10}, since the value $\theta(t) = 2-\chi(t)$ is called the \emph{extremal coefficient} for $t\in\RR^d$. 
As our choice for the name suggests, the tail correlation function $\chi$ is a symmetric positive definite function. It was proposed as an extreme value analogue to the correlation function \cite{schlathertawn_03} and is generally considered an appropriate summary statistic for extremal behaviour of stationary processes, {\cf \cite{davismikosch_09,colesheffernantawn_99,beirlantetalii_03,falk_05,fasenetalii_10} among many others}.
Since $\eta:=1-\chi$ satisfies the triangle inequality $\eta(s \pm t) \leq \eta(s) + \eta(t)$, the TCF $\chi$ cannot be differentiable except when $\chi$ is constant \cite{markov_95,schlathertawn_03}.
Estimators can be found for instance in \cite{smith_90} (raw estimates) or \cite{schlathertawn_03,cooleyetalii_06,naveauetalii_09}.

Here, we explore for the first time to what extent the TCF can
distinguish between different classes of max-stable processes. In fact, we identify practically relevant, but essentially different, stationary max-stable processes on $\RR^d$ sharing the same TCF (Section~\ref{sec:tcf}).  The focus lies on stationary max-stable processes, and particular emphasis is put on radially symmetric TCFs that are monotonously decreasing as the radius grows.

The text is structured as follows: 
After the introductory Section~\ref{sec:intro}, where some notation is fixed, Section~\ref{sec:maxstable} gives an overview over well-known classes of stationary max-stable processes. The main contribution is  Section~\ref{sec:tcf}, where we compare the TCFs of these classes and identify systematic co-occurrences. Section~\ref{sec:opscounter} complements Section~\ref{sec:tcf} in that it provides counterexamples of TCFs that cannot arise from certain classes of processes. Thereby, we transfer two well-known operations from Geostatistics to the class of TCFs.
The text closes with a short Section~\ref{sec:parametric} on parametric families of TCFs with sharp parameter bounds for being a TCF. All proofs are postponed to Section~\ref{sec:proofs}.

\paragraph{Some notation}  By $a \wedge b$ we denote the minimum between two quantities $a$ and $b$, whereas $\Max_{i \in I} a_i$ is the supremum over the $a_i$.  The function $\eins_{A}$ is the indicator function of $A$. The expression $\nu_d$ stands for the Lebesgue measure on the Borel $\sigma$-algebra $\Borel^d$ of $\RR^d$ and $\lVert \cdot \rVert$ is the Euclidean norm on $\RR^d$.
We denote 
\begin{align*}B^d_{r} :=\{h \in \RR^d \pmid \lVert h \rVert \leq r \}\end{align*}
the $d$-dimensional ball of radius $r$ centred at the origin $o \in \RR^d$.
The constant 
\begin{align*}\kappa_d:=\vol_d(B^d_{1})=\pi^{d/2}/\Gamma(1+d/2)\end{align*} 
is the volume of the $d$-dimensional unit ball. When a function on $\RR^d$ depends on the radius (Euclidean norm on $\RR^d$) only, we will usually  treat it  as a function on $[0,\infty)$. The expression cdf abbreviates ``cumulative distribution function''. When we treat a cdf $G$ on $(0,\infty)$, it is always meant that $G(0+)=0$. Usually $G(0+)=c \in [0,1]$ also yields admissible models, but will lead to a mixture with trivial components.
The function  
\begin{align*}
\erfc(x)=\frac{2}{\sqrt{\pi}}\int_{x}^\infty e^{-y^2} \,\D y
\end{align*} 
is the complementary error function. We write $\erf(x):=1-\erfc(x)$ for the error function.


\section{Max-stable processes}\label{sec:maxstable}

A stochastic process $X=\{X_t\}_{t \in \RR^d}$ on $\RR^d$ is called \emph{max-stable} if all its finite-dimensional distributions are max-stable, that is, for each $m,n \in \NN$, $t_1,\dots,t_m \in \RR^d$ and $n$ independent copies $(Y^{(i)})_{i=1}^n$  of the random vector $Y:=(X_{t_1},\dots,X_{t_m})$ we have 
\begin{align*}
  \Max_{i=1}^n Y^{(i)} \stackrel{\distr}{=} a_n Y + b_n
\end{align*}
for suitable norming sequences $(a_n)_{n \in \NN}$ and $(b_n)_{n \in \NN} $ with values in $\RR^m$ and $a_n > 0$. All operations here are meant componentwise, and $\stackrel{\distr}{=}$ means equality in distribution. 
In what follows, we will consider \emph{stationary} max-stable processes  on $\RR^d$.  Since non-degenerate one-dimensional marginal distributions will be considered only,
we henceforth restrict ourselves to \emph{standard Fr{\'e}chet marginals}, \ie $\PP(X_t \leq x)=e^{-1/x}$ for $t \in \RR^d$ and $x > 0$ (\cf \cite{resnick_08}), whereas plots of simulated processes will be always transformed to \emph{standard Gumbel marginals}, \ie $\PP(X_t \leq x)=e^{-e^{-x}}$ for $t \in \RR^d$ and $x\in \RR$.  
It has been shown (\cf \cite{kabluchko_09,dehaan_84,stoev_08}) that max-stable processes that are separable in probability allow for a \emph{spectral representation} of the following form 
\begin{align}\label{eqn:spectralrep}
\{X_t\}_{t \in \RR^d} \stackrel{\distr}{=} \left\{ \Max_{n=1}^\infty U_n V_t(\omega_n) \right\}_{t \in \RR^d}.
\end{align}
Here, $(U_n,\omega_n)$ denotes an (enumerated) Poisson point process on $\RR_+ \times \Omega$ with intensity $u^{-2} \D u \times \nu(\D \omega)$ for some measure space $(\Omega, {\mathcal A}, \nu)$, and  $V_t:\Omega \rightarrow \RR_+$ is measurable  with $\int_{\Omega} V_t(\omega) \nu(\D \omega) = 1$ for each $t \in \RR^d$. The functions  $\{V_t\}_{t \in \RR^d}$  are called \emph{spectral functions}. 
Of course, any process $X$ of the form (\ref{eqn:spectralrep}) is max-stable and has standard Fr{\'e}chet marginals.
In terms of a spectral representation the finite-dimensional distributions of $X$ are given through
\begin{align}\label{eqn:fddspectral}
-\log \PP(X(t_k)\leq x_k;\, k=1,\dots,m) = \int_{\Omega}  \Max_{k=1}^m  \frac{V_{t_k}(\omega)}{x_k} \,\nu(\D \omega)
\end{align}
and the TCF $\chi$ of the max-stable process $X$ may be expressed as
\begin{align}\label{eqn:chispectral}
\chi(t):=\int_\Omega V_t(\omega) \wedge V_o(\omega) \,\nu(\D \omega).
\end{align}
If the measure space $(\Omega, {\cal A}, \nu)$ is a probability space, the spectral functions $\{V_t\}_{t \in \RR^d}$ themselves form a stochastic process on $\RR^d$, which we will call  \emph{spectral process}. It is convenient then to interpret the expression $\{V(\omega_n)\}_{n=1}^\infty$ in the spectral representation (\ref{eqn:spectralrep})  as \iid sequence $V^{(n)}$ of a process $V=\{V_t\}_{t \in \RR^d}$ on $\RR^d$ that is independent of the Poisson point process $\{U_n\}_{n=1}^\infty$ on $\RR_+$.


\subsection{Examples of stationary max-stable processes}\label{sec:examplesMSTprocesses}
The following processes $X$ on $\RR^d$ are stationary and max-stable. They have either been proposed in previous literature or constitute modifications or extensions of those. 
Note that the stationarity of the spectral process $V$ is a sufficient but not a necessary condition for $X$ being stationary (\cf \cite{kabluchkoetalii_09,molchanovstucki_13}).

\paragraph{(Mixed) Moving Maxima (M3/M2) and subclasses (M3r, M2r and M3b)} 
Slightly different notions are given in the literature, \cf \cite{kabluchkostoev_12, schlather_02, stoev_08, stoevtaqqu_06, smith_90}, for example. We consider the following normalized version: Let $\{f(t)\}_{t \in \RR^d}$ be a measurable process on $\RR^d$ with values in $[0,\infty]$, such that 
\begin{align}\label{eqn:Lonecondition}
\EE_f \left( \int_{\RR^d} f(t) \D t \right) = 1.
\end{align} 
As in \cite{oestingschlather_12,engelkeetalii_12} we refer to  $f \in \FF$ as \emph{(random) shape function} or \emph{(random) storm}. Further, we consider the measure space 
\begin{align*}
(\Omega,{\mathcal A},\nu)=(\RR^d\times \FF, \Borel^d \otimes {\mathcal F}, \nu_d \times \PP_f),
\end{align*} 
where 
$(\FF, {\mathcal F},\PP_f)$ stands for the law of the random shape function $f$.
Then we call the process $X$ with spectral representation (\ref{eqn:spectralrep}) given by $(\Omega,{\mathcal A},\nu)$ and spectral functions
\begin{align*}
  V_t\left((z,f)\right)&:=f(t-z), \qquad (z,f)\in \RR^d \times \FF, \qquad t \in \RR^d,
\end{align*}
\emph{Mixed Moving Maxima process (M3 process)}, or \emph{Moving Maxima process (M2 process)} if $f$ is deterministic, respectively.

We put particular emphasis on such random storms, where each realization of a random shape $f\geq 0$ is \emph{radially symmetric} around the origin $o \in \RR^d$ and \emph{non-increasing} as the radius grows, and refer to this class as \emph{M3r processes}, or \emph{M2r processes} if $f$ is deterministic, respectively. Moreover, we will also consider the subclass of \emph{M3b processes} where the M3 process has as shape functions only \emph{normalized indicator functions of balls $B^d_R$}, \ie  
\begin{align*}
f(t)=  \frac{\eins_{B^d_R}(t)}{\vol_d(B^d_R)} = \frac{\eins_{B^d_R}(t)}{\kappa_d R^d}  
\end{align*}
with a random radius $R \in (0,\infty)$. Clearly, M3r, M2r and M3b processes are stationary and isotropic.

\paragraph{Mixed Poisson storm processes (MPS)}
Here, we consider a mixed version of the Poisson storm process introduced in \cite{lantuejoulbacrobel_11}. 
Before we define the process, let us make some preliminary considerations {(with terminology from stochastic geometry based on \cite{schneiderweil_08}).}
If $C$ is the \emph{typical cell} of a stationary isotropic \emph{Poisson hyperplane mosaic} of \emph{intensity}~$1$ 
and $\pintensity > 0$, then $\pintensity^{-1}  C =\{x : \pintensity x \in C \}$ is distributed like the typical cell corresponding to the intensity $\pintensity$ and has expected volume 
\begin{align}
\EE\left(\nu_d\left({\pintensity}^{-1} C \right)\right) = \frac{d^d \kappa_d^{d-1}}{\kappa_{d-1}^d \pintensity^d}  =: \frac{1}{\mu_d(\pintensity)}
\end{align}
(\cf \cite[(10.4) and (10.4.6)]{schneiderweil_08}). Note that our notion of intensity $\pintensity$ is based on \cite[\ppage 497 and \page 126]{schneiderweil_08} and corresponds to the choice $\lambda=\pintensity \kappa_{d-1}/ (d\kappa_d)$ with $\lambda$ as in \cite[\page 420]{lantuejoulbacrobel_11}.

Now, let $\pintensity \in (0,\infty)$ be a random variable distributed according to a distribution function $F$ on $(0,\infty)$ (with $F(0+)=0$). Let $C$ be the typical cell of a stationary isotropic Poisson hyperplane mosaic of intensity $1$ that is independent of $\pintensity$ and set 
\begin{align}\label{eqn:Poissonstorm}
  f(t):= \mu_d(\pintensity) \, \eins_{{\pintensity}^{-1}C }(t),  \qquad t \in \RR^d.
\end{align}
Conditioning on $\pintensity$, one sees that, indeed, $f$ satisfies (\ref{eqn:Lonecondition}) and, thus, defines an M3 process $X$ with standard Fr{\'e}chet marginals, which is stationary and isotropic. We call this process \emph{Mixed Poisson storm process (MPS process)} with \emph{intensity mixing distribution}~$F$.

\paragraph{(Variance-mixed) Brown-Resnick processes (BR and VBR)}
Let $\{W_t\}_{t \in \RR^d}$ be a Gaussian process with {stationary increments} (meaning that the law of $\{W_{t+h}-W_h\}_{t \in \RR^d}$ does not depend on $h \in \RR^d$) and variance $\sigma^2(t)=\Var(W_t)$.
Then we call the process $X$ defined through the spectral process
\begin{align*}
  V_t=\exp\left({W_t-\frac{\sigma^2(t)}{2}}\right),  \qquad t \in \RR^d,
\end{align*}
\emph{Brown-Resnick process (BR process)}. 
The process $X$ is stationary and its law depends on the variogram $\gamma(t)=\EE(W_t-W_o)^2$ only  \cite[Theorem~2]{kabluchkoetalii_09}.     

We will also consider a mixture of BR processes with respect to the variance of the involved Gaussian process. 
As in the construction of BR processes let $\{W_t\}_{t \in \RR^d}$ be a  Gaussian process with {stationary increments} and variance $\sigma^2(t)$. Additionally, let  $S$ be an independent random variable on $(0,\infty)$ with distribution function $G$ (with $G(0+)=0$). Then we call the process $X$ with spectral process
\begin{align*}
  V_t=  \exp\left(S W_t -\frac{S^2 \sigma^2(t)}{2} \right),  \qquad t \in \RR^d,
\end{align*}
\emph{variance-mixed Brown-Resnick process} with \emph{variance mixing distribution} $G$.
The law of $X$ is also stationary and it depends on the variogram $\gamma(t)=\EE(W_t-W_o)^2$ and the distribution function $G$ only.

\begin{remark} 
A similar construction can be found in \cite{engelkekabluchko_12}, where the BR process is mixed in its scale instead, \ie
\begin{align*}
  V_t=  \exp\left(W_{St} -\frac{\sigma^2(St)}{2} \right),  \qquad t \in \RR^d.
\end{align*}
This yields in fact the same class of processes in the most prominent example when $W_t$ is a fractal Brownian motion and, thus, self-similar, such that the law of $(W_{ct})_{t \in \RR^d}$ coincides with the law of $(\lvert c\rvert ^{\alpha/2} W_t)_{t \in \RR^d}$ for $c\neq 0$ and some $\alpha \in (0,2)$.
\end{remark}

\paragraph{Extremal Gaussian and extremal binary Gaussian processes (EG and EBG)}
Here, we relate to \cite[Theorem~2]{schlather_02}. Let $Z=\{Z_t\}_{t \in \RR^d}$ 
be a stationary Gaussian process whose marginals follow a standard normal distribution. The correlation function of $Z$ will be denoted by $\rho(t)$. 
Based on $Z$, we call the process $X$ defined through the spectral process
\begin{align*}
  V_t= \sqrt{2\pi} \cdot (Z_t)_+,   \qquad t \in \RR^d,
\end{align*}
\emph{extremal Gaussian process (EG process)} (where $z_+=\max(z,0)$).  
Secondly, we call the process $X$ with spectral process
\begin{align}\label{eqn:EBGspectral}
  V_t=2 \cdot \eins_{\left\{Z_t > 0\right\}},   \qquad t \in \RR^d,
\end{align}
\emph{extremal binary Gaussian process (EBG process).}

\section{Co-occurrence of tail correlation functions}\label{sec:tcf}

\begin{table}\small
\centering
\begin{tabular}{p{6mm} p{25mm}  p{26mm}  >{$\displaystyle}p{43mm}<{$} p{25mm}}
\toprule\\[-3mm]
\multicolumn{2}{l}{\textbf{Process model}} &  \textbf{Parameter} & \text{\textbf{TCF} $\chi(t)$ for $t \in \RR^d$} & \textbf{Reference}\\[1mm] 
\toprule\\[-3mm]
M3r & \multirow{1}{25mm}{M3 of radial non-increasing shapes}  & \multirow{1}{25mm}{non-increasing\\ random shape\\ $f\geq 0$ on $[0,\infty)$: \\ $\EE_f \rVert f\lVert_{L^1(\RR^d)} = 1$} &  \EE_f \int_{\RR^d} f(\lVert z \rVert) \wedge f(\lVert z-t \rVert) \,\D z  & Eqn.\ (\ref{eqn:chispectral})\\[14mm]
M2r & \multirow{1}{25mm}{M2 of radial non-increasing shapes}  & \multirow{1}{25mm}{non-increasing\\ determ. shape \\$f\geq 0$ on $[0,\infty)$: \\ $\rVert f\lVert_{L^1(\RR^d)} = 1$} &  \int_{\RR^d} f(\lVert z \rVert) \wedge f(\lVert z-t \rVert) \,\D z  & \\[14mm]
M3b & \multirow{1}{25mm}{M3 of ball indicators}  & \multirow{1}{25mm}{random radius\\ $R$ on $(0,\infty)$} &  \EE_R \int_{\RR^d}  \frac{\eins_{\lVert z\rVert \leq R} \wedge \eins_{\lVert z-t\rVert \leq R}}{\kappa_d R^d}\,\D z 
& \\[5mm]
MPS & \multirow{1}{25mm}{Mixed Poisson\\ Storm}  & \multirow{1}{25mm}{cdf $F$ on $(0,\infty)$ } &   \laplacetransform(F) \left( \frac{2\kappa_{d-1}}{d\kappa_{d}} \lVert t \rVert \right) & 
\multirow{1}{25mm}{\cite[Prop.\ 4]{lantuejoulbacrobel_11} \\ with $K=\{o,t\}$}
\\[5mm]
\midrule\\[-3mm]
BR & Brown-Resnick  & variogram $\gamma$ & \erfc\left(\sqrt{\gamma(t)/8}\right) & \cite[Rk.\ 25]{kabluchkoetalii_09} \\[5mm]
VBR & \multirow{1}{25mm}{Variance-mixed\\ Brown-Resnick} & \multirow{1}{25mm}{variogram $\gamma$,\\ cdf $G$ on $(0,\infty)$} & \multirow{1}{5cm}{$\displaystyle{\int_0^\infty\erfc\left(s \,\sqrt{\gamma(t)/8}\right) \,\D G(s)}$}
 & \\[8mm]
\midrule\\[-3mm]
EG & \multirow{1}{25mm}{extremal\\ Gaussian}  & correlation $\rho$ &  1 - \sqrt{(1-\rho(t))/2} & \cite[Eqn.\ (7)]{cooleyetalii_06} \\[5mm]
EBG & \multirow{1}{25mm}{extremal binary\\ Gaussian} & correlation $\rho$ &   {\pi}^{-1} \arcsin \rho(t) + 1/2 & 
\multirow{1}{27mm}{\cite[Eqn.\ (10.8.3)]{cramerleadbetter_67}\\ with $u=0$}
\\[5mm]
\bottomrule
\end{tabular}
\caption{Tail correlation functions $\chi(t)$ for $t \in \RR^d$ of stationary max-stable processes on $\RR^d$ from Section~\ref{sec:examplesMSTprocesses}. The process models are grouped according to different long-range dependence. 
Here $\erfc(x)=2 \pi^{-1/2}\int_{x}^\infty e^{-y^2} dy$ 
denotes the complementary error function and 
$\laplacetransform(F)(x)=\int_0^\infty \exp(-xt) \D F(t)$ 
denotes the Laplace transform of the distribution function $F$. 
}\label{table:semiparameterTCFs}
\end{table}


The TCFs of the max-stable processes from Section~\ref{sec:examplesMSTprocesses} are listed in Table~\ref{table:semiparameterTCFs}.  The formulae are given in the indicated references or can easily be derived from the contributions therein.  
We want to explore to what extent the TCF can distinguish between these classes of processes. To this end we use the notation
\begin{align*}
T^d_{\text{\emph{model}}} := \left\{ \chi: \RR^d \rightarrow [0,1] \,\, \middle| \,\, \begin{array}{l} \chi \text{ TCF of a process $X$ on $\RR^d$} \\ \text{from the process class \emph{model}} \end{array} \right\}
\end{align*}
when referring to the set of TCFs of a certain class of processes on $\RR^d$.
For instance, $T^d_{\text{M3}}$ is the set of TCFs of M3 processes on $\RR^d$. By 
\begin{align*}
T^d := \left\{ \chi: \RR^d \rightarrow [0,1] \,\, \middle| \,\, \begin{array}{l} \chi \text{ TCF on $\RR^d$}  \end{array} \right\}
\end{align*}
we simply denote the set of \emph{all} TCFs on $\RR^d$. 
As a first observation we note that 
\begin{align*}
T^d_{\EG} \cap T^d_{\MMM} = \emptyset
\qquad \text{ and } \qquad 
T^d_{\EBG} \cap T^d_{\MMM} = \emptyset
\end{align*}
due to the different behaviour towards long-range dependence.
While M3 processes are shown to be {mixing} 
\cite{stoev_08,kabluchkostoev_12}, and
EG and EBG processes feature long-range dependence \cite{kabluchkostoev_12,wangroystoev_13},
BR processes may entail both behaviours depending on the variogram. If the variogram defining the BR process tends to $\infty$ fast enough, a BR process may even be representable as an M3 process   \cite[Theorem~14]{kabluchkoetalii_09}. 
The different ergodic behaviour is also reflected in the behaviour of the TCF $\chi(t)$ as $t$ tends to $\infty$ (\cf \cite{kabluchkoschlather_10,wangroystoev_13}). 
Accordingly, we will henceforth treat mixing processes and non-ergodic processes separately.


{
\paragraph{Absolute and complete monotonicity} The subsequent considerations rely on certain  monotonicity properties of functions. Therefore, we introduce the following notions in advance \cite[Chapter IV]{widder_46}. A real-valued function $f$ is  \emph{completely monotone} (\resp \emph{absolutely monotone}) on an interval $I$  if it has derivatives of all orders on the interior {$\interior{I}$} with $(-1)^k f^{(k)}(x) \geq 0$ (\resp $f^{(k)}(x) \geq 0$) for all $x \in {\interior{I}}$ and $k\in \NN \cup \{0\}$ and if additionally $f$ is continuous at the boundary points of $I$.  In the literature, the focus often lies on the intervals $I=(0,\infty)$ or $I=[0,\infty)$, since completely monotone functions on $[0,\infty)$ are precisely the continuous functions $f$ such that $f(\lVert \cdot \rVert^2)$ is positive definite on $\RR^d$ for all dimensions $d$. Such functions are characterized as Laplace transforms of non-decreasing functions or, equivalently, positive measures, \cf \cite[Theorem~12, Chapter IV]{widder_46} for example. 
}

\subsection{Mixing processes}\label{sec:M3_BR}

Here, we restrict ourselves to stationary and \emph{isotropic} processes on $\RR^d$ and focus on the subclass of BR and VBR processes that are associated to variograms that are radially symmetric around the origin $o \in \RR^d$ and grow  \emph{monotonously} to $\infty$ as the radius grows. Secondly, we involve the M3 processes from Table~\ref{table:semiparameterTCFs}. 
M2r processes and M3b processes each form a proper subclass of M3r processes. However, their TCFs even coincide in every dimension. We refer to Section~\ref{sec:proofs} for all proofs.

\begin{proposition}\label{prop:M3coincide}
\begin{enumerate}[a)]
\item 
For all $d\geq 1$ we have $T^d_{\MMMr}=T^d_{\MMr}=T^d_{\MMMb}$.
\item
In the equality $T^d_{\MMr}=T^d_{\MMMb}$ the deterministic shape function $f$ of an M2r process and the distribution function $H$ of $1/R$, where $R$ is the random radius of an M3b process, can be recovered from each other by 
\begin{align}\label{eqn:fandH}
f(u)=\frac{1}{\kappa_d} \int_0^{1/u} s^d \D H(s) \qquad \text{and} \qquad H(s)=\kappa_d \int_0^s \frac{1}{u^d} \,\D \left[f\left(\frac{1}{u}\right)\right].
\end{align}
\end{enumerate}
\end{proposition}

In fact, the class $T^d_{\MMMr}$ is well-known in Geostatistics and has been intensively studied in \cite{gneiting_99} (therein called $H_d$). Thus, we can benefit from Gneitings analysis, which is based on \cite{williamson_56} and characterizes $T^d_{\MMMr}$ by monotonicity properties. In particular
\begin{align}\label{eqn:ToneMMMr}
T^1_{\MMMr} &= 
\left\{ \chi: [0,\infty) \rightarrow [0,1] \,\, \middle| \,\, \begin{array}{l}  \text{$\chi$ continuous, convex,}  \\ \text{$\chi(0)=1$ and $\lim_{t \to \infty}\chi(t)=0$} \end{array} \right\}. 
\end{align}
The precise characterization of $T^d_{\MMMr}$ for $d\geq 2$ in terms of convexity properties is stated in \cite[Theorem~3.1.\ and 3.3.]{gneiting_99}. 
Moreover, \cite{gneiting_99} gives inversion formulae that we use to recover the defining quantities $f$ and $R$ of the classes $T^d_{\MMr}$ and $T^d_{\MMMb}$, respectively. 
The explicit expressions in dimensions $d=1,2,3$ are given in Table~\ref{table:recoverfGg} and derived in Section~\ref{sec:tableproofs} (Proof of Table~\ref{table:recoverfGg}). This is of special interest to us when we want to simulate the corresponding processes for a given TCF $\chi$. 

The classes $T^d_{\MMMr}$ are all nested, \ie $T^d_{\MMMr} \supset T^{d+1}_{\MMMr}$ for all $d \in \NN$. Gneiting \cite{gneiting_99} also characterizes the class
\begin{align}\label{eqn:Hinfty}
T^\infty_{\MMMr} := \bigcap_{d=1}^\infty T^d_{\MMMr}
\end{align}
as scale mixtures of the complementary error function \cite[Theorems 3.7 and 3.8]{gneiting_99}
\begin{align}\label{eqn:TinftyMMMr}
T^\infty_{\MMMr} 
&=\left\{ \varphi(t) = \int_{(0,\infty)} \erfc(st) \,\D G(s)  \,:\, \text{$G$ cdf on $(0,\infty)$}\right\}\\
\notag &=\left\{ \varphi: [0,\infty) \rightarrow \RR \,:\, 
\begin{array}{l}
\text{\upshape $\varphi$ continuous, $-\varphi'\left(\sqrt{\cdot}\right)$ completely monotone on $(0,\infty)$,} \\
\text{$\varphi(0)=1$ and $\lim_{t \to \infty}\varphi(t)=0$} 
\end{array}
\right\},
\end{align}
which entails the following characterization of TCFs of VBR processes (\cf Table~\ref{table:semiparameterTCFs}):
\begin{align}\label{eqn:TdVBR}
T^d_{\VBR} &= \left\{ \varphi\left(\sqrt{\gamma/8}\right) \,:\, \text{{\upshape $\gamma$  variogram on $\RR^d$ and $\varphi \in T^\infty_{\MMMr}$}}\right\}.
\end{align}
In Table~\ref{table:erfcmix} we give some examples of corresponding pairs $\varphi$ and distribution functions $G$ (or probability densities $g=G'$) that we need to know in order to simulate a VBR process with prescribed TCF $\chi=\varphi(\sqrt{\gamma/8})$.

Finally, we observe that in every dimension   the class of TCFs arising from MPS processes is given by Laplace transforms of cdfs  on $(0,\infty)$ and, thus, coincides with 
\begin{align*}
T^d_{\MPS} &= \left\{ \psi: [0,\infty) \rightarrow \RR \,:\, \text{\upshape $\psi$ completely monotone, $\psi(0)=1$ and $\lim_{t \to \infty}\psi(t)=0$} \right\}.
\end{align*}
In particular, the class $T^d_{\MPS}$ does not depend on the specific dimension $d$, even though the involved factor ${2\kappa_{d-1}}/{(d\kappa_{d})}$ in Table~\ref{table:semiparameterTCFs} does. These obervations lead to the following inclusions of the classes of TCFs arising from mixing processes, which are also illustrated in Figure~\ref{fig:M3_BR}.

\begin{proposition}\label{prop:MixingInclusions} The following inclusions hold for all dimensions $d\geq 1$:
\begin{enumerate}[a)]
\item\label{item:MPSinMMMr} $T^d_{\MPS} \subset T^\infty_{\MMMr} \subset T^d_{\MMMr}$.
\item $T^d_{\BR} \cup T^\infty_{\MMMr} \subset T^d_{\VBR}$.
\item\label{item:erfcCondition} $\erfc(t^\alpha) \in T^d_{\MPS} \cap T^d_{\BR}$ $\Leftrightarrow$ $\alpha \in (0,0.5]$. In particular $T^d_{\MPS} \cap T^d_{\BR} \neq \emptyset$.
\end{enumerate}
\end{proposition}

\setlength{\unitlength}{1cm}
\begin{table}
{\small
\centering
\begin{tabular}{p{.05\textwidth} | >{$\displaystyle}p{.15\textwidth}<{$}  >{$\displaystyle}p{.40\textwidth}<{$}  >{$\displaystyle}p{.22\textwidth}<{$}}
\toprule \\[-3mm]
& d=1 & d=2 & d=3 \\[1mm] 
\toprule \\[-1mm]
$f(u)$ 
&  -\chi'(2u)    
&  \frac{4u}{\pi} \int_0^{1/(2u)} ((2ut)^{-2}-1)^{1/2} \,\, \D \lambda_{\chi}(t)
& \chi''(2u)/(\pi u) 
\\[5mm]
$k(s)$ 
&  s \chi''\left(s\right) 
&  \frac{s^2}{2} \int_0^{1/s} \left((s/t)^2-1\right)^{-1/2} \,\, \D \lambda_{\chi}(t) 
&  \frac{s}{3}\left(\chi''(s)-s \chi'''(s)\right)
\\[5mm]
\bottomrule
\end{tabular}
\caption{
Recovery expressions for the defining quantities of M2r and M3b processes from a given TCF $\chi \in T^d_{\MMMr} = T^d_{\MMr} = T^d_{\MMMb}$ in dimensions $d=1,2,3$ (\cf Proposition~\ref{prop:M3coincide}): (i) the monotone shape function $f$ of an M2r process and (ii) the density $k$ of $2R$, where $R$ is the random radius that defines an M3b process (if the density $k$ exists). 
The functions $f$ and $k$ are defined on $(0,\infty)$ where $f$ may have a pole at $0$ and $k$ may have other poles as well.
We abbreviate $\lambda_{\chi}(t):=t \chi''(1/t)$. 
In case $d=2$ we assume here that even $\chi \in T^5_{\MMMr}$ holds in order to eliminate an additional integral.
}\label{table:recoverfGg}
\vspace{0.5cm}
}
{\small
\begin{tabular}{>{$\displaystyle}p{.52\textwidth}<{$}  >{$\displaystyle}p{.27\textwidth}<{$}  >{$\displaystyle}p{.12\textwidth}<{$}}
\toprule \\[-3mm] 
\text{Distribution function } G(s) \text{ or } g(s)=G'(s) & 
\multicolumn{2}{ >{$\displaystyle}p{.36\textwidth}<{$} }{\varphi(t)=\int_{0}^\infty \erfc(st) \, \D G(s)} \\[4mm] 
\toprule \\[-2mm]
G(s)=e^{-1/{(as)^2}}  & e^{-2t/a} & 0 < a \\[3mm]
g(s)=\frac{\sqrt{\pi}}{\Gamma(\nu)\Gamma(\frac{1}{2}-\nu)} \int_0^s 
\frac{x^{2\nu-3} e^{-1/(4x^2)}}{\left(s^2-x^2\right)^{\nu+1/2}} \, \D x
& \frac{2^{1-\nu}}{\Gamma(\nu)}\, t^{\nu} K_\nu(t) &  0 < \nu < \frac{1}{2} \\[5mm]
G(s)=\erf(as)  & 1-\frac{2}{\pi} \arctan\left(t/a \right) & 0 < a\\[3mm]
G(s)=1-e^{-(as)^2}  & 1-\left(1+\left(t/a\right)^{-2}\right)^{-1/2} & 0 < a\\[3mm]
\bottomrule
\end{tabular}
\caption{Members of the class $T^\infty_{\MMMr}$ (\cf (\ref{eqn:Hinfty})) and their corresponding distribution function $G(s)$ or probability density $g(s)=G'(s)$ on $(0,\infty)$ as scale mixtures of the complementary error function. As special cases the exponential model, the Whittle-Mat{\'e}rn family, the arctan model and the Dagum model (\cf \cite{bergmateuporcu_08}) appear. 
}\label{table:erfcmix}
}
\end{table}

\setlength{\unitlength}{1cm}
\begin{figure}
\footnotesize
\centering
\begin{picture}(14,7)
\put(0,1){\framebox(10.5,6)[tl]{\hspace{0.3cm} \begin{minipage}{7cm} \vspace{0.3cm}  {\normalsize $T^d_{\MMMr} = T^d_{\MMr} = T^d_{\MMMb}$} \\[1mm] \textsf{(M3 processes of radial non-increasing shapes)}  \end{minipage}}}
\put(1,1.5){\framebox(9,3.5)[tl]{\hspace{0.3cm} \begin{minipage}{7cm} \vspace{0.3cm} {\normalsize $T^\infty_{\MMMr}$} \end{minipage}}}
\put(4.5,2){\framebox(5,2.5)[tl]{\hspace{0.3cm} \begin{minipage}{7cm} \vspace{0.3cm} {\normalsize $T^d_{\MPS}$} \textsf{(Mixed Poisson Storms)}\\[1mm] = \textsf{completely monotone TCFs} \end{minipage}}}
\put(1.8,0.5){\dashbox{0.1}(11.8,2.3)[tr]{\begin{minipage}{2.4cm} \vspace{0.3cm} \hfill {\normalsize $T^d_{\BR}\,\,$} \\[1mm] \textsf{(Brown-Resnick)}\\ \textsf{with $\gamma$ radial,}\\ \textsf{increasing to $\infty$}  \end{minipage}\hspace{0.3cm}}}
\put(0.5,0){\dashbox{0.1}(13.5,5.5)[tr]{\begin{minipage}{2.4cm} \vspace{0.3cm} \hfill {\normalsize $T^d_{\VBR}\,\,$} \\[1mm] \textsf{(Variance-mixed}\\ \textsf{Brown-Resnick)}\\ \textsf{with $\gamma$ radial,}\\ \textsf{increasing to $\infty$}  \end{minipage}\hspace{0.3cm}}}
\put(2.2,2.2){\textbf{(Prop.\ \ref{prop:MixingNonempty}\ref{item:BRohneMPS})}}
\put(2.2,3.5){\textbf{(Prop.\ \ref{prop:MixingNonempty}c)}}
\put(6.2,2.3){\textbf{(Prop.\ \ref{prop:MixingInclusions}\ref{item:erfcCondition})}}
\put(8.2,6.2){\textbf{(Prop.\ \ref{prop:MixingNonempty}\ref{item:M3ohneVBR})}}
\end{picture}
\caption{
Inclusions and intersection of sets of tail correlation functions arising from mixing max-stable processes as stated in Propositions \ref{prop:M3coincide} and \ref{prop:MixingInclusions}. Propositions \ref{prop:MixingInclusions} \ref{item:erfcCondition}) and \ref{prop:MixingNonempty} 
show that the indicated regions are non-empty.
}\label{fig:M3_BR}
\vspace{1cm}
\hspace{-1cm}\textsf{\textbf{\scriptsize BR process}}
\hspace{4cm} \textsf{\textbf{\scriptsize MPS process}}\\
\mbox{
\scalebox{0.33}{\includegraphics{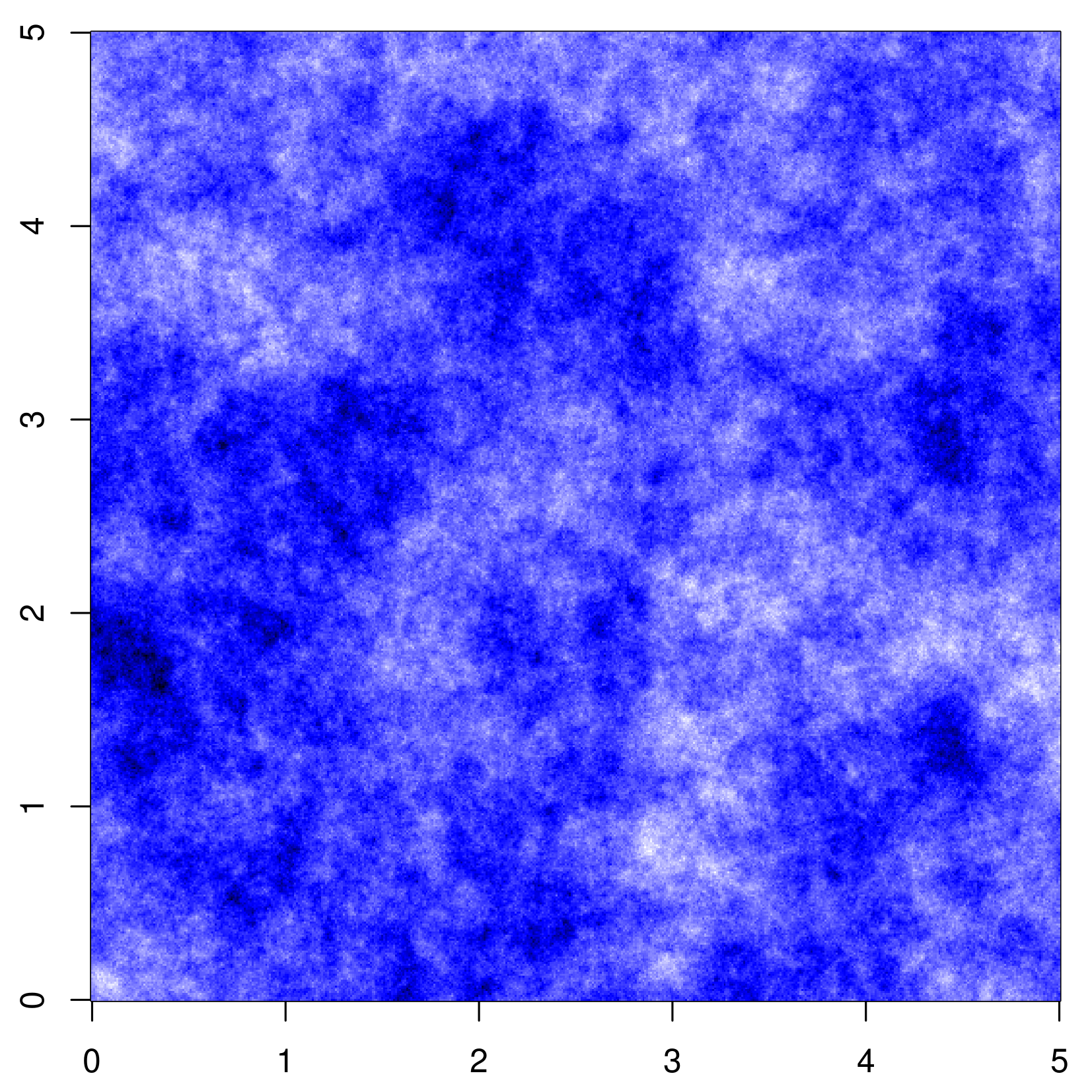}}
\scalebox{0.33}{\includegraphics{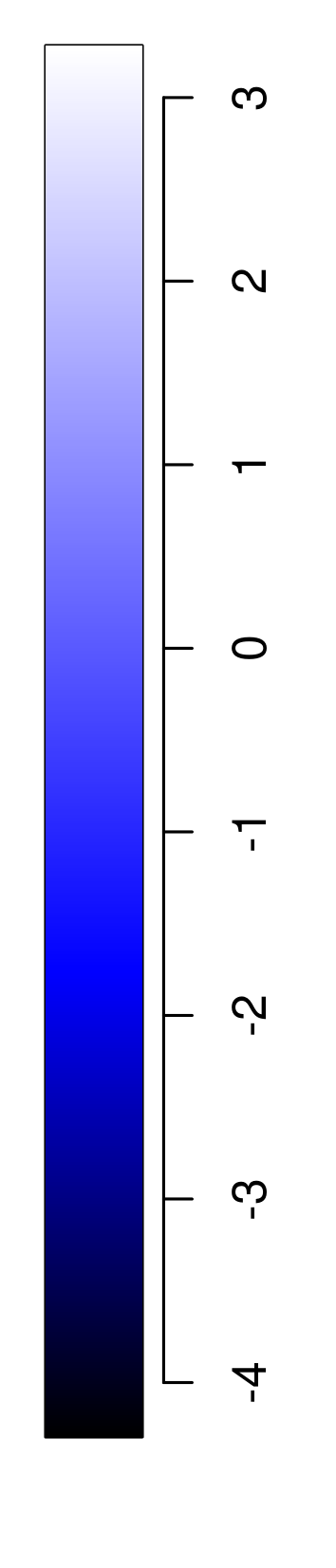}}
\scalebox{0.33}{\includegraphics{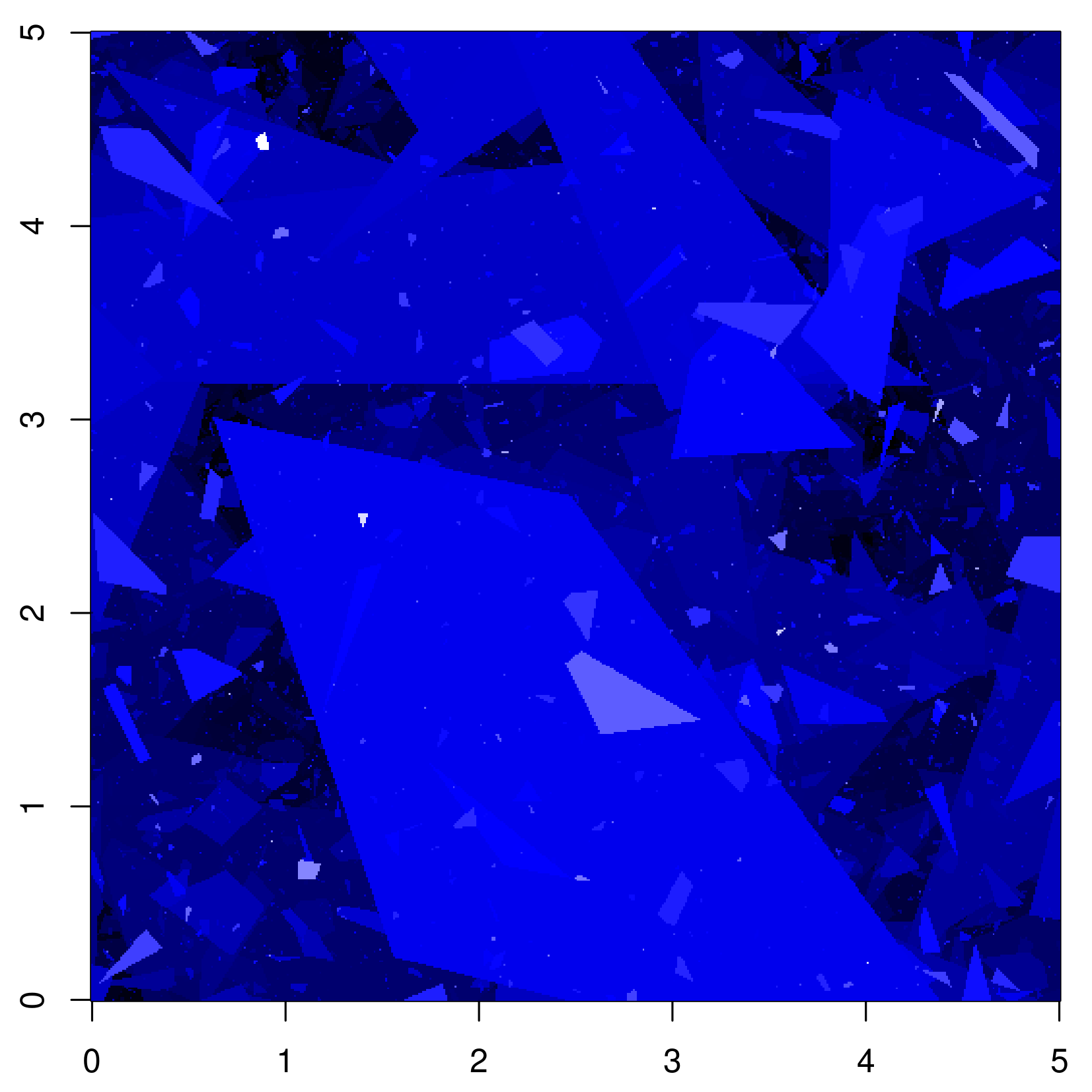}}
\scalebox{0.33}{\includegraphics{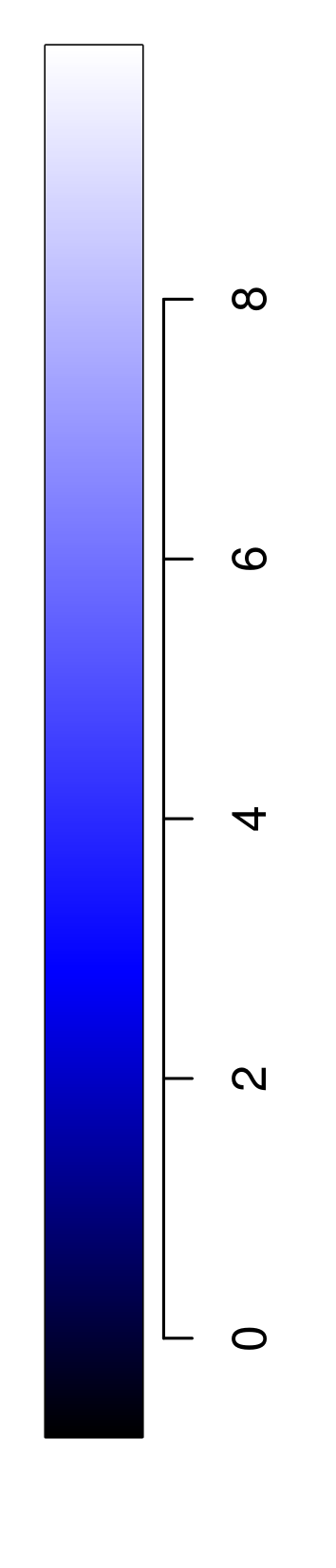}}
}
\mbox{
\scalebox{0.33}{\includegraphics{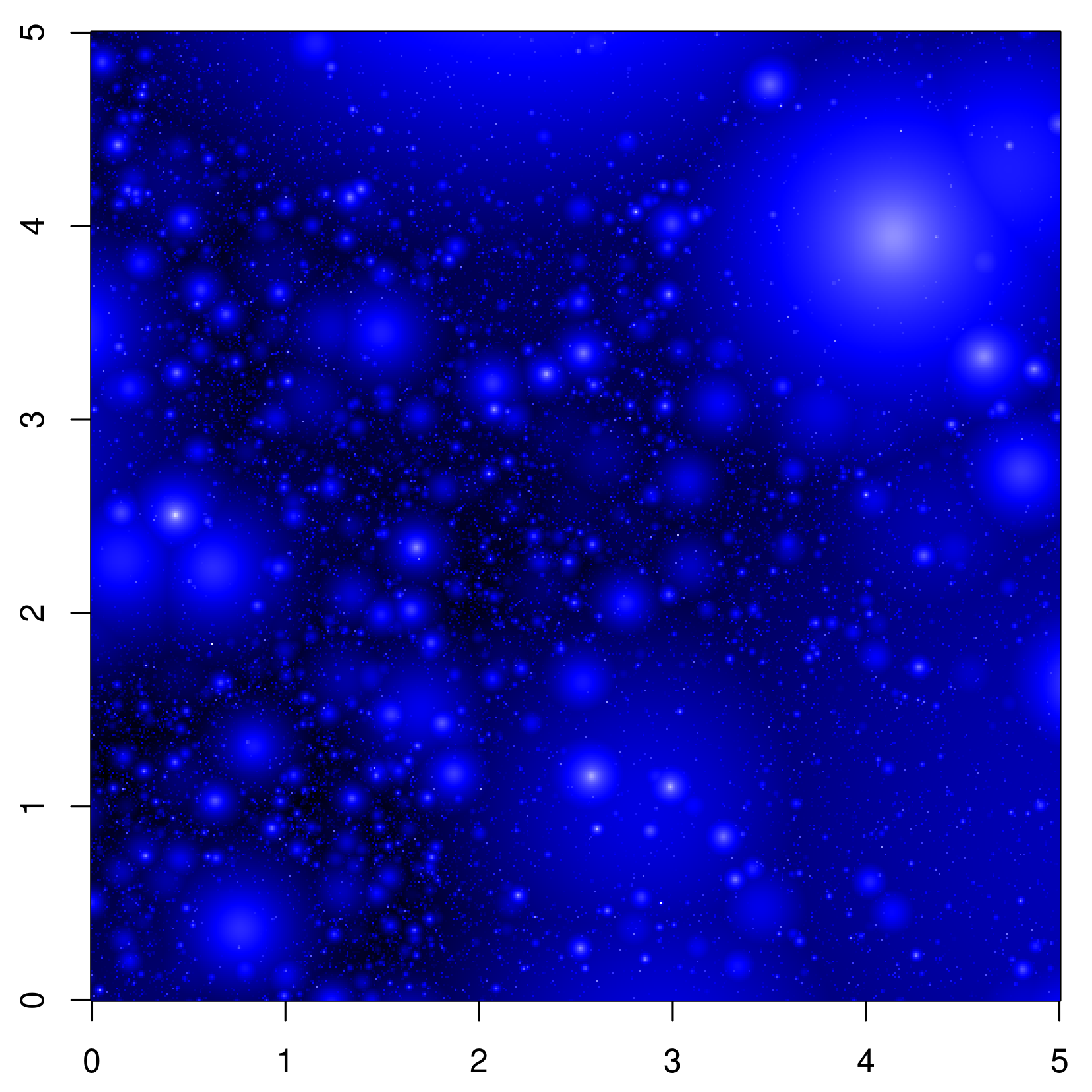}}
\scalebox{0.33}{\includegraphics{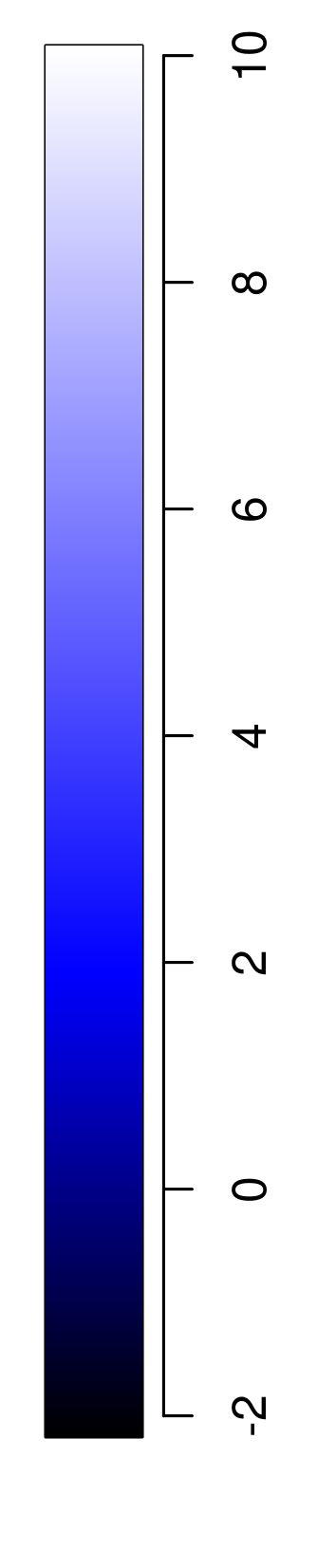}}
\scalebox{0.33}{\includegraphics{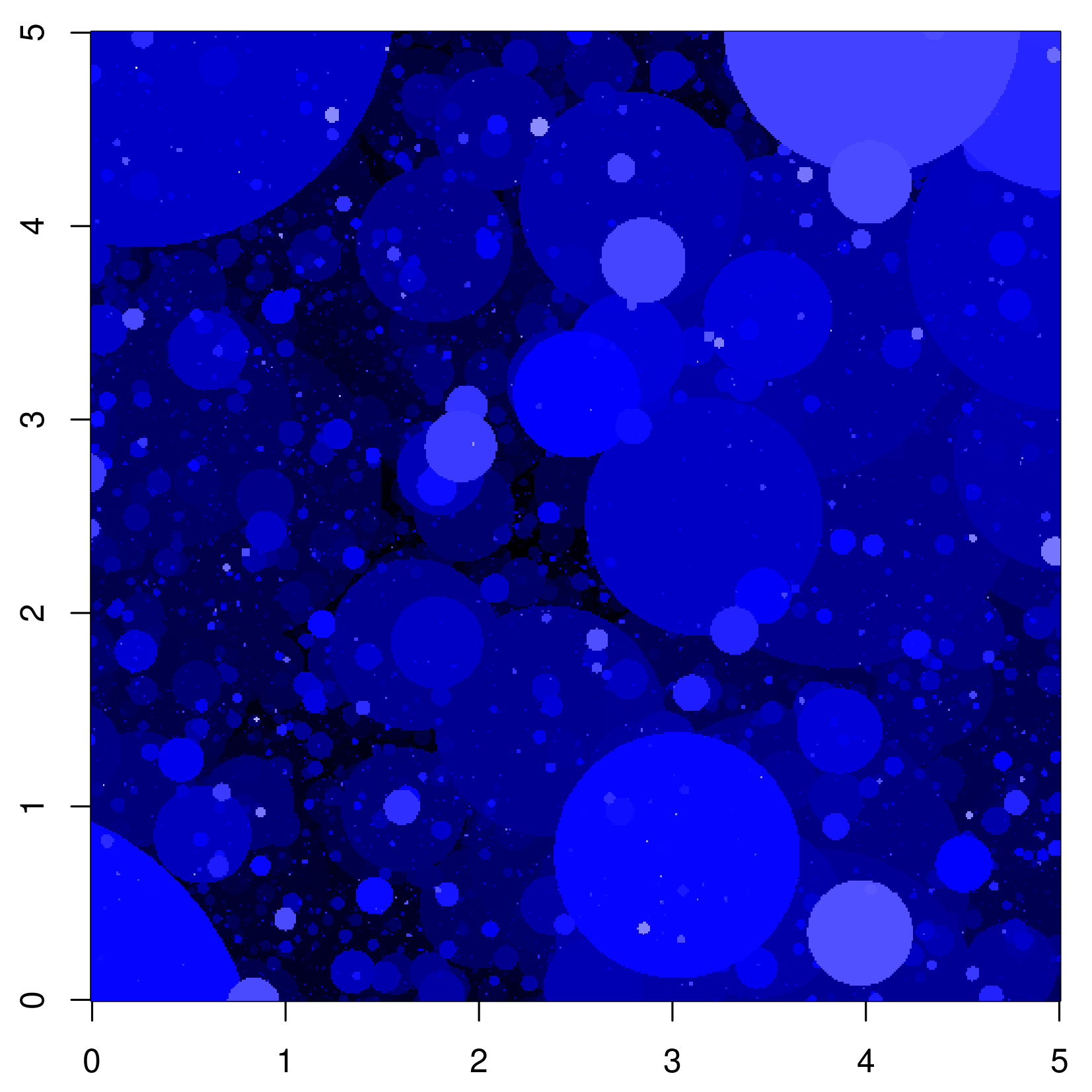}}
\scalebox{0.33}{\includegraphics{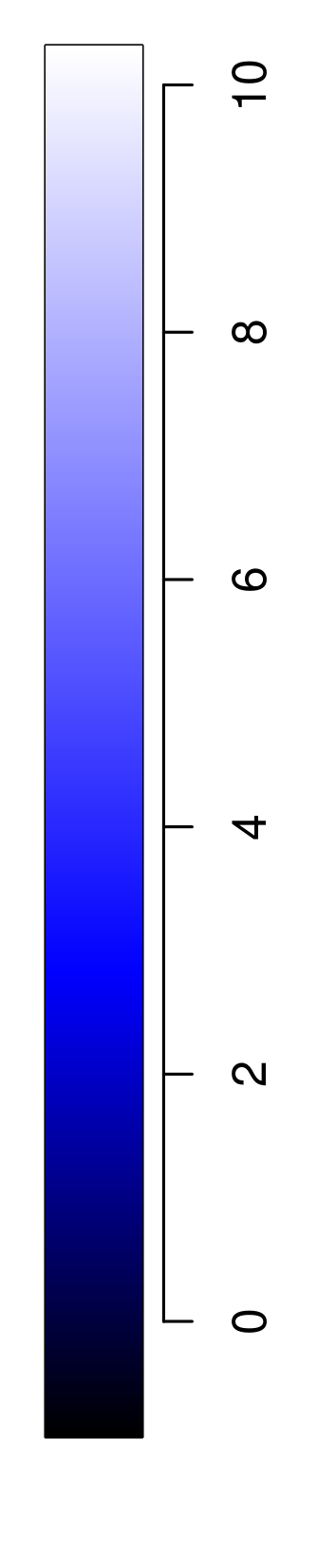}}
}\\
\hspace{-1cm}\textsf{\textbf{\scriptsize M2r process}}
\hspace{4cm} \textsf{\textbf{\scriptsize M3b process}}
\caption{Simulations of different processes on $[0,5]^2 \subset \RR^2$ 
with identical tail correlation function $\chi(t)=\erfc(\sqrt{\lVert t\rVert})$ (see Example~\ref{example:M3shape_M3ball_PS_BR}): Brown-Resnick process (BR), Mixed Poisson Storm process (MPS), two-dimensional section of an M2r process with deterministic shape (M2r), two-dimensional section of an M3b process of normalized ball indicator functions (M3b). The plots were transformed to standard Gumbel marginals.} \label{fig:M3shape_M3ball_PS_BR}
\end{figure}

\begin{example}\label{example:M3shape_M3ball_PS_BR}
We consider the following four processes on $\RR^2$:
\begin{enumerate}[(i)]
\item the BR process on $\RR^2$ associated to the variogram 
\begin{align*}\gamma(t)=8 \lVert t \rVert, \qquad t \in \RR^2,\end{align*}
\item the MPS process on $\RR^2$ with intensity mixing distribution
\begin{align*}F(s) = 
\left\{
\begin{array}{ll}
0 &\qquad \text{if } s \leq \frac{\pi}{2},\\
2\pi^{-1} \arctan\left(\sqrt{\frac{2s}{\pi}-1}\right) & \qquad \text{if } s>\frac{\pi}{2}, 
\end{array}
\right.
\end{align*} 
\item the restriction of the M2r process on $\RR^3$ with deterministic shape function
\begin{align*}
f(t)= \frac{1+4\lVert t\rVert}{ (\pi)^{3/2} \lVert 2 t \rVert^{5/2}} e^{-2\lVert t\rVert},
\qquad t \in \RR^3, 
\end{align*}
to $\RR^2 = \{ (t_1,t_2,0) : t \in \RR^3 \}$,
\item the restriction of the M3b process on $\RR^3$ where the density $k$ of $2R$ is given by 
\begin{align*}
k(s)= \frac{\left(4 s^2 + 8s +5\right)  e^{-s}}{12 \sqrt{\pi s}}, \qquad s \in [0,\infty)
\end{align*} 
to $\RR^2 = \{ (t_1,t_2,0) : t \in \RR^3 \}$.
\end{enumerate}
Then all of these processes on $\RR^2$ share the same TCF 
\begin{align*}
\chi(t)=\erfc\left(\sqrt{\lVert t \rVert}\right), \qquad t \in \RR^2.
\end{align*}
The variogram $\gamma$ corresponds to Brownian motion. Proposition~\ref{prop:MixingInclusions}\ref{item:erfcCondition}) ensures that the cdf $F$ exists, and Propositions \ref{prop:MixingInclusions}\ref{item:MPSinMMMr}) and \ref{prop:M3coincide} give the existence of $f$ and $R$. While the recovery of $F$ follows from \cite[\page 1100 17.13.5]{gradshteynryzhik_07}, the quantities $f$ and $R$  are recovered from $\chi$ as in Table~\ref{table:recoverfGg}. For ease of simulation we consider only the two-dimensional sections of M2r and M3b processes on $\RR^3$ instead of two-dimensional M2r and M3b processes.
Figure~\ref{fig:M3shape_M3ball_PS_BR} shows simulations of the BR process and the restricted M2r and M3b processes that were obtained using the R-package RandomFields V3.0 {\cite{RandomFields_3.0}}. 
\end{example}



\subsection{Non-ergodic processes}\label{sec:Gaussian_BR}

Here, we compare the TCFs of EG, EBG and BR processes. Therefore, we need to consider BR processes associated to bounded variograms on $\RR^d$. Such variograms $\gamma$ are always of the form 
\begin{align*}
\gamma(t)=\lambda (1-\rho(t)), \qquad t \in \RR^d,
\end{align*}
where $\rho$ is a correlation function on $\RR^d$, and $\lambda > 0$  (\cf \cite[Section~3.1]{gneitingsasvarischlather_01} or \cite[\page 32]{chilesdelfiner_99}). 
Hence, the TCFs of BR, EG and an EBG processes all depend on a correlation function  $\rho$ (\cf Table~\ref{table:semiparameterTCFs}). The ansatz  
\begin{align*}
\chi_{\EG}(t)&=\chi_{\EBG}(t) \quad & \Leftrightarrow \quad &&
1 - \sqrt{(1-\rho_{\EG}(t))/2} 
&=  {\pi}^{-1} \arcsin \rho_{\EBG}(t) + 1/2\\
\chi^{(\lambda)}_{\BR}(t)&=\chi_{\EG}(t) \quad & \Leftrightarrow \quad &&
\erfc\left[\sqrt{\lambda \left(1-\rho_{\BR}(t)\right) / 8}\right] 
&=  1 - \sqrt{(1-\rho_{\EG}(t))/2} \\
\chi^{(\lambda)}_{\BR}(t)&=\chi_{\EBG}(t) \quad & \Leftrightarrow \quad &&
\erfc\left[\sqrt{\lambda \left(1-\rho_{\BR}(t)\right) / 8}\right] 
&= {\pi}^{-1} \arcsin \rho_{\EBG}(t) + 1/2,
\end{align*}
leads to the question, if or for which $\lambda > 0$  the maps
\begin{align}
\label{eqn:R}
R: &[-1,1] \rightarrow [-1,1], \qquad & R(x) &:= \cos \left( \pi \, \sqrt{\left(1-x\right) / 2} \right) \\
\label{eqn:Slambda}
S_\lambda: &[-1,1] \rightarrow [-1,1], \qquad & S_\lambda(x) &:= 1- 2 \left(\erf\left[\sqrt{\lambda \left(1-x \right) / 8}\right]\right)^2 \\
\label{eqn:Tlambda}
T_\lambda: &[-1,1] \rightarrow [-1,1], \qquad & T_\lambda(x)&:=\cos \left( \pi \, \erf\left[\sqrt{\lambda \left(1-x\right) / 8}\right] \right)
\end{align}
(or its inverses $R^{-1}$, $S_\lambda^{-1}$, $T_\lambda^{-1}$) transform correlation functions again into correlation functions. 

\begin{proposition}\label{prop:cortrafo}
Let $A \in \{R,S_\lambda,T_\lambda\}$, where $R$, $S_\lambda$ and $T_\lambda=R \circ S_\lambda$ are the maps from (\ref{eqn:R}),(\ref{eqn:Slambda}) and (\ref{eqn:Tlambda}). Then $A$ transforms a correlation function of the form 
\begin{align*}
\rho= (1-\alpha) \rho_0 + \alpha  \quad \text{where $\rho_0$ is a correlation function and $\alpha \in [0,1]$}
\end{align*} 
 again into a correlation function if
\begin{align*}
\alpha &\geq \frac{1}{2} && \qquad \text{in case $A=R$,}\\
\lambda &\leq \frac{8 (\erf^{-1}(1/\sqrt{2}))^2}{1-\alpha} \approx \frac{4.425098}{1-\alpha} && \qquad \text{in case $A=S_\lambda$,}\\
\lambda &\leq \frac{8 (\erf^{-1}(1/2))^2}{1-\alpha} \approx \frac{1.8197}{1-\alpha} && \qquad \text{in case $A=T_\lambda$.}
\end{align*}  
\end{proposition}

{
Thus, we have the systematic intersections of classes of TCFs as illustrated in Figure~\ref{fig:Gaussian_BR}. 
Note that $T_\lambda = R \circ S_\lambda$ and the upper bound on $\lambda$ in Proposition~\ref{prop:cortrafo} for $T_\lambda$ is smaller than the upper bound for $S_\lambda$, such that the transformation $T_\lambda$ gives rise only to elements in the intersection of all three classes of TCFs.
}

\begin{example}\label{example:EG_EBG_BR}
The BR process on $\RR^d$ associated to the variogram 
\begin{align*}\gamma_{\BR}(t)=1.62 (1-\exp(-\lVert t \rVert)), \qquad t \in \RR^d,\end{align*}
the EG process on $\RR^d$ associated to the correlation function 
\begin{align*}\rho_{\EG}(t)=1- 2 \left(\erf\left[0.45 \sqrt{1- \exp(-\lVert t \rVert)}\right]\right)^2, \qquad t \in \RR^d,\end{align*} 
and the EBG process on $\RR^d$ associated to the correlation function 
\begin{align*}\rho_{\EBG}(t)=\cos \left( \pi \, \erf\left[0.45 \sqrt{1-\exp(-\lVert t \rVert)}\right] \right), \qquad t \in \RR^d, \end{align*} 
all share the same TCF 
\begin{align*}\chi(t)=\erfc\left[0.45 \sqrt{1-\exp(-\lVert t \rVert)}\right], \qquad t \in \RR^d.\end{align*}
Indeed $\gamma_{\BR}$ is a well-known variogram on $\RR^d$ (\cf \eg \cite[Section~4]{gneitingsasvarischlather_01}) and Proposition~\ref{prop:cortrafo} ensures that the functions $\rho_{\EG}$ and $\rho_{\EBG}$ are correlation functions on $\RR^d$, such that the respective processes are well-defined. 
Figure~\ref{fig:EG_EBG_BR} shows simulations of these processes in dimension $d=2$ that were obtained using the R-package RandomFields V3.0 {\cite{RandomFields_3.0}}.
\end{example}

\setlength{\unitlength}{1cm}
\begin{figure}\footnotesize
\centering
\begin{picture}(14,5)
\put(0.5,1.5){\framebox(9.5,3.5)[tl]{\hspace{0.3cm} \begin{minipage}{4cm} \vspace{0.3cm} {\normalsize $T^d_{\EG}$} \textsf{(extremal Gaussian)} \\[1mm] \textsf{with} {\small $\rho_{\EG} = (1-\alpha)\rho_0 + \alpha$} \end{minipage}}}
\put(0,0){\framebox(9.5,3)[bl]{\hspace{0.3cm} \begin{minipage}{7cm}  {\normalsize $T^d_{\EBG}$} \textsf{(extremal binary Gaussian) with} {\small $\rho_{\EBG}$} \\  \end{minipage}}}
\put(5.3,1){\framebox(8.7,3.5)[tr]{\begin{minipage}{3.3cm} \vspace{0.3cm} {\normalsize \hfill $T^d_{\BR}$} \textsf{(Brown-Resnick)} $\,\,$\\[1mm] $\phantom{abcde}$ \textsf{with}  {\small \begin{align*}\gamma_{\BR}&=\lambda (1-\rho_{\BR}) \\ \rho_{\BR}&=(1-\beta)\rho_{1}+\beta \end{align*}}  \end{minipage}\hspace{0.3cm}}}
\put(1.2,2.2){\begin{minipage}{6cm}{{\small $\rho_{\EBG}=R(\rho_{\EG})$ \\[1mm] \textsf{for} $\alpha \geq 0.5$}}\end{minipage}}
\put(5.8,2.2){\begin{minipage}{6cm}{{\small $\rho_{\EBG}=T_\lambda(\rho_{\BR})$ \\[1mm] \textsf{for} $\lambda \leq 1.8197/(1-\beta)$}}\end{minipage}}
\put(5.8,3.7){\begin{minipage}{6cm}{{\small $\rho_{\EG}=S_\lambda(\rho_{\BR})$ \\[1mm] \textsf{for} $\lambda \leq 4.425098/(1-\beta)$}}\end{minipage}}
\end{picture}
\caption{Intersections of sets of tail correlation functions $\chi$ arising from max-stable processes with long-range dependence as obtained from Proposition~\ref{prop:cortrafo}. The equations in the intersections show that the respective regions are non-empty and  how to transform the defining quantities of the respective processes. Here $R$, $S_\lambda$ and $T_\lambda=R \circ S_\lambda$ are the maps from (\ref{eqn:R}),(\ref{eqn:Slambda}) and (\ref{eqn:Tlambda}). 
} \label{fig:Gaussian_BR}
\vspace{1cm}
\mbox{
\scalebox{0.33}{\includegraphics{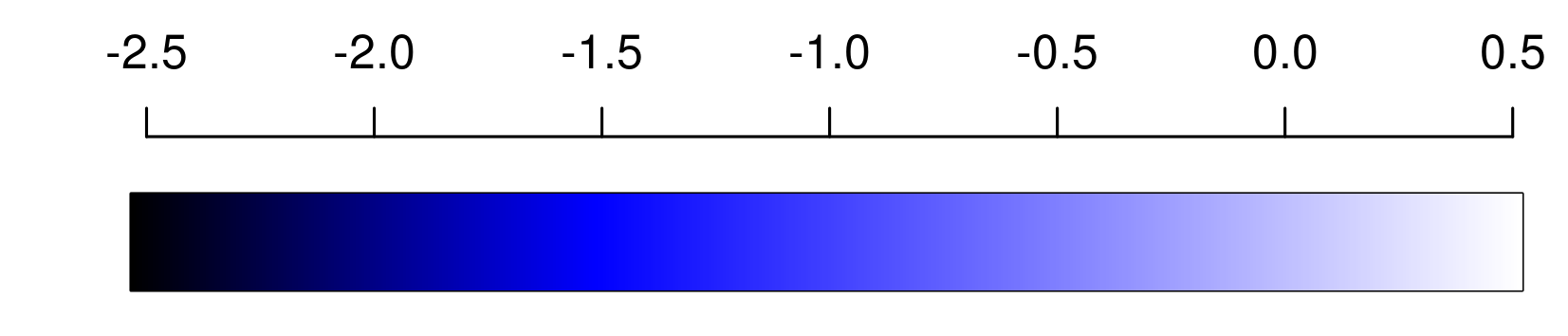}}
\scalebox{0.33}{\includegraphics{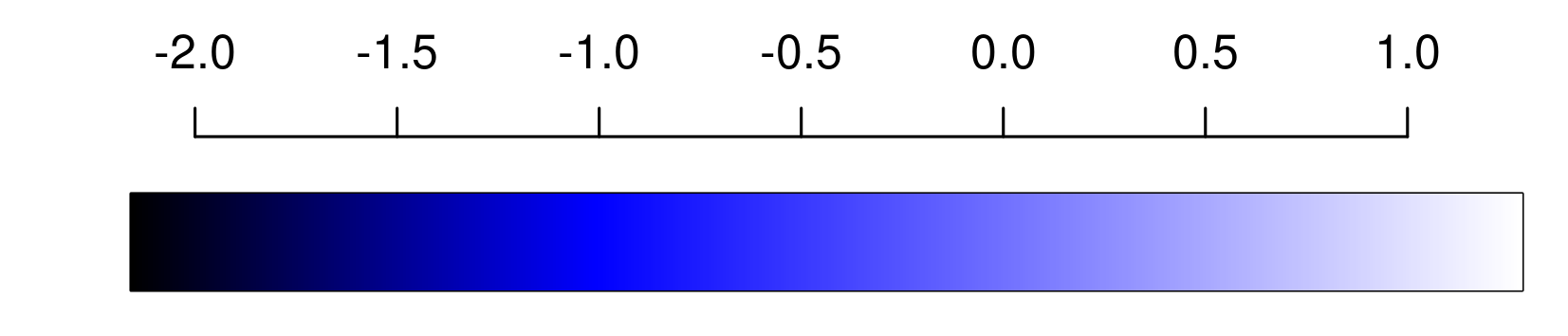}}
\scalebox{0.33}{\includegraphics{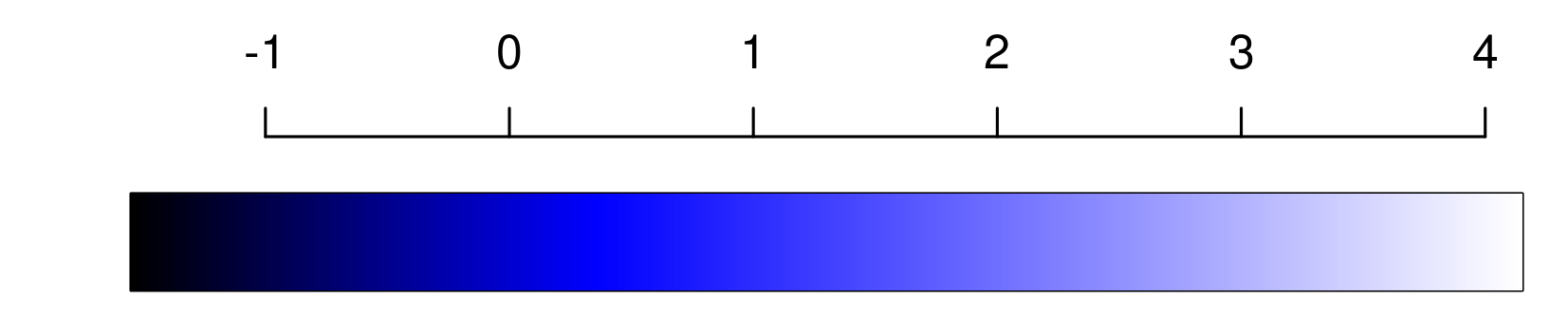}}
}
\mbox{
\scalebox{0.33}{\includegraphics{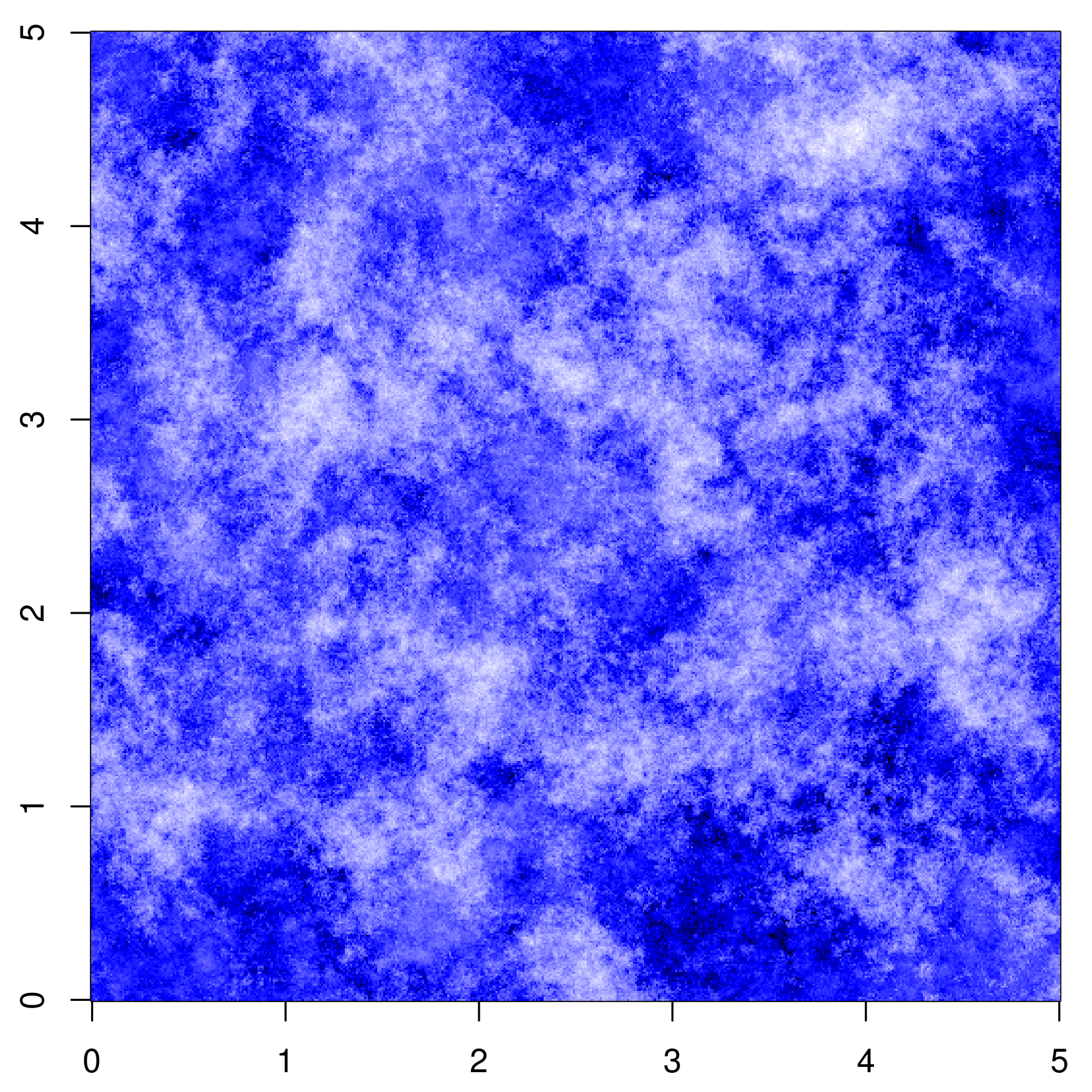}}
\scalebox{0.33}{\includegraphics{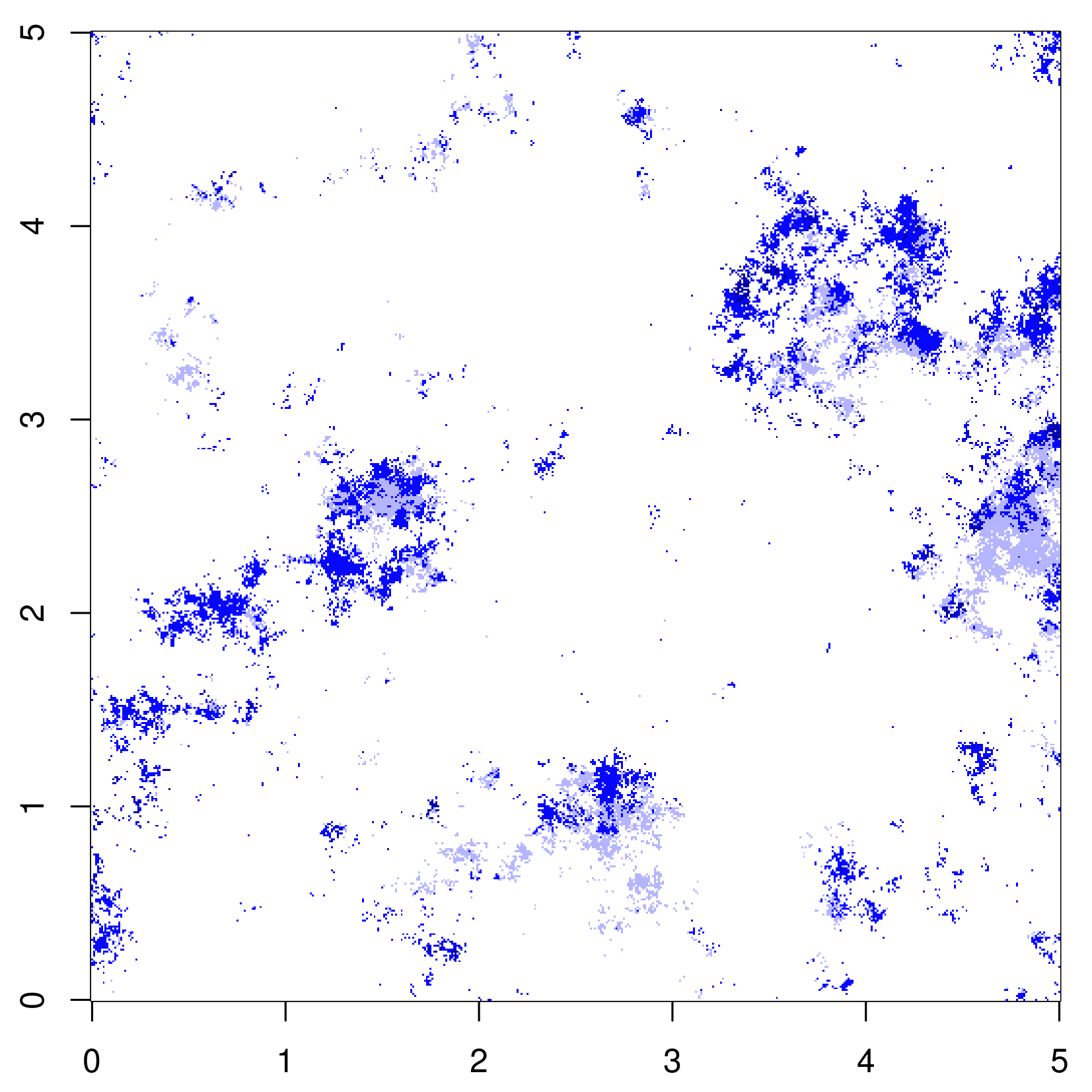}}
\scalebox{0.33}{\includegraphics{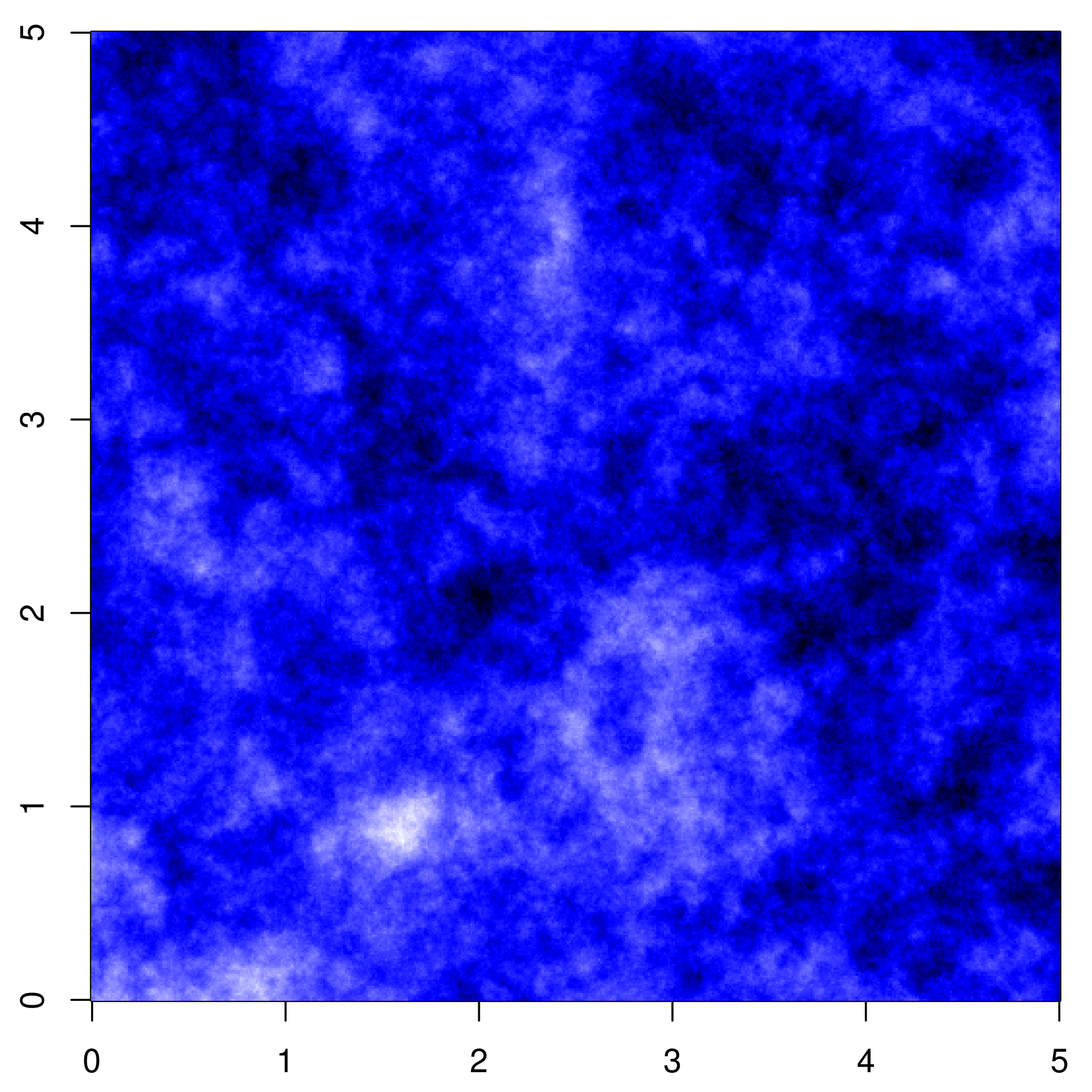}}
}
\hspace{2.5cm} 
\textsf{\textbf{\scriptsize EG process}}
\hspace{2.9cm} 
\textsf{\textbf{\scriptsize EBG process}}
\hspace{2.8cm}
\textsf{\textbf{\scriptsize BR process}}
\caption{Simulations of different processes on $[0,5]^2 \subset \RR^2$ 
with identical tail correlation function $\chi(t)=\erfc(0.45 (1-\exp(-\lVert t \rVert))^{1/2}))$ (see Example~\ref{example:EG_EBG_BR}): extremal Gaussian process (EG), extremal binary Gaussian process (EBG) and Brown-Resnick process (BR).
The plots were transformed to standard Gumbel marginals.
Note that simulations of the EBG process are bound to take their maximum value at a high proportion of the total area, \cf their spectral process (\ref{eqn:EBGspectral}).
} \label{fig:EG_EBG_BR}
\end{figure}

\section{Counterexamples}\label{sec:opscounter}

While the previous Section~\ref{sec:tcf} is concerned with inclusions and intersections of classes of TCFs,
this section provides statements of the form $T^d_{model1} \setminus T^d_{model2} \neq \emptyset$. Exemplarily, we address only mixing processes in this section, \cf Figure~\ref{fig:M3_BR}.

\begin{proposition}\label{prop:MixingNonempty}
We have for all dimensions $d\geq 1$:
\begin{enumerate}[a)]
\item\label{item:BRohneMPS} $\erfc(t^\alpha) \in T^d_{\BR}$ $\Leftrightarrow$ $\alpha \in (0,1]$. In particular, $T^d_{\BR} \setminus T^d_{\MPS} \neq \emptyset$. 
\item\label{item:M3ohneVBR} The class $T^d_{\VBR}$ does not contain functions with compact support.\\ In particular $T^d_{\MMMr} \setminus T^d_{\VBR} \neq \emptyset$.
\end{enumerate}
There exists a dimension $d_0 \in \NN$, such that for all $d \geq d_0$
\begin{enumerate}[c)]
\item\label{item:M3inftyohneBR} 
$T^\infty_{\MMMr} \setminus T^d_{\BR} \neq \emptyset$.
\end{enumerate}
\end{proposition}

Moreover,  one might get the impression that any continuous radial TCF on $\RR^d$ that is non-increasing and convex on $[0,\infty)$ and that vanishes at $\infty$ belongs already to the class $T^d_{\MMMr}$ or at least appears already in Figure~\ref{fig:M3_BR}. This is true for $d=1$ since $T^1_{\MMMr}$ comprises all of these functions. The following two operations, however, yield counterexamples for $d \geq 3$, as we shall see in Section~\ref{sec:cexamples}.
Firstly, the \emph{turning bands operator} has been inspired by \cite{kabluchkostoev_12} and is well-known in the context of isotropic Gaussian processes.
Secondly, the \emph{multiplication with the class} $T^d_{\MMMr}$ can shorten the range of tail dependence to a compact set.
Both operations are derived from construction principles for the corresponding max-stable processes that can be applied to (almost arbitrary) spectral representations.

\subsection{Turning bands}

\paragraph{The turning bands operator} 
Let $k,d \in \NN$ with $1\leq k \leq d$. The set of ordered tuples $(x_1,\dots,x_k)$ of $k$ orthonormal vectors in $\RR^d$ is known as the \emph{Stiefel manifold of orthonormal $k$-frames in $\RR^d$} (\cf \eg \cite[\page 131]{nachbin_76}) and denoted $V_k(\RR^d)$. If we interpret the vectors $x_1,\dots,x_k$ as columns of a matrix, we identify
\begin{align}\label{eqn:Stiefel}
V_k(\RR^d)=\{ A \in \RR^{d \times k} \pmid A^\trans A = \Eins_{k\times k} \},
\end{align}
where $A^\trans$ denotes the transpose of $A$ and $\Eins_{k\times k}$ the identity matrix in $\RR^{k \times k}$. A matrix $A \in V_k(\RR^d)$ embeds $\RR^k$ linearly and isometrically into $\RR^d$, whereas $A^\trans$ applied to a vector $t \in \RR^d$ is a vector in $\RR^k$ whose coordinates can be interpreted as the coordinates of the projection of $t$ onto $A(\RR^k)$ with respect to the orthonormal frame defined by the columns of $A$. For $k=1$ the Stiefel manifold is simply the sphere $V_1(\RR^d)=S^{d-1}$, and for $k=d$ the orthogonal group $V_d(\RR^d)=\orth(d)$.
In view of (\ref{eqn:Stiefel})  the Stiefel manifold $V_k(\RR^d)$ is a {compact submanifold} of $\RR^{d \times k}$. The action of the orthogononal group $\orth(d)$ (from the left) exhibits $V_k(\RR^d)$ as a {locally compact homogeneous space} on which a unique normalized left invariant \emph{Haar measure} $\sigma^d_k$ can be defined \cite[\page 142 Example~4]{nachbin_76}, which we call \emph{uniform distribution} \cite{mardiakhatri_77,juppmardia_79}. 

By $C(\RR^d)$ (resp.\ $C(\RR^k)$) we denote the set of real-valued continuous functions on $\RR^d$ (resp.\ $\RR^k$).  Since $V_k(\RR^d)$ is compact, the following operator $\tb^d_k$, which we call \emph{turning bands operator}, is well-defined:
\begin{align*}
\tb^d_k: C(\RR^k) \rightarrow C(\RR^d), \qquad \tb^d_k(f)(t):=\int_{V_{k}(\RR^d)} f\left(A^\trans(t)\right) \, \sigma^d_k(\D A).
\end{align*}
Moreover, it is compatible with compositions (see Lemma~\ref{lemma:Stiefelcomposition})
\begin{align}\label{eqn:TBcomposition}
\tb_{k_1}^{k_2} \circ \tb_{k_2}^{k_3} = \tb_{k_1}^{k_3} \qquad \text{for $k_1 \leq k_2 \leq k_3$.}
\end{align}
In the context of Gaussian processes and positive definite functions, the turning bands operator $\tb^d_1$ is a familiar operator, see \cite{matheron_73,gneiting_99a,gneiting_99b,schlather_12,zucastell_02}, where explicit formulae and recurrence relations are provided. 
Let $\Phi_d$ denote the set of radially symmetric continuous correlation functions on $\RR^d$. Then it is well-known that $\tb_1^d$ yields a bijection between $\Phi_1$ and $\Phi_d$. In view of (\ref{eqn:TBcomposition}) this implies that $\tb_k^d$ is a bijection between $\Phi_k$ and $\Phi_d$. The operator $\tb_k^d$ for arbitrary $k,d \in \NN$ with $k \leq d$ is usually implicitly addressed as $\tb_1^d \circ (\tb_1^k)^{-1}$ in the references above. Because of these bijections, the \emph{turning bands method} is an important tool for the simulation of stationary isotropic Gaussian processes.
In the context of max-stable processes and their TCFs the situation transfers to the following extent.

\paragraph{The turning bands method for max-stable processes} 

Let $X$ be a stochastically continuous simple max-stable process on $\RR^k$. 
Then the process $X$ has a spectral representation as in (\ref{eqn:spectralrep})
\begin{align}\label{eqn:baseprocess}
X_t= \Max_{n=1}^\infty U_n V_t(\omega_n), \qquad t \in \RR^k,
\end{align}
where $(U_n,\omega_n)$ denotes a Poisson point process on $\RR_+ \times \Omega$ with intensity $u^{-2} \D u \, \nu(\D\omega)$ and the spectral function $V_t(\omega)$ is \emph{jointly measurable} in the variables $t \in \RR^k$ and $\omega \in \Omega$. Based on this representation we define another simple max-stable process $Y$ on $\RR^d$ with $d \geq k$ as follows.
Let  $(U_n,\omega_n,A_n)$ be a Poisson point process on $\RR_+ \times \Omega \times V_k(\RR^d)$  of intensity $u^{-2}\D u \, \nu(\D \omega)\, \sigma^d_k(\D A)$, where  $\sigma^d_k(\D A)$ is the uniform distribution on the Stiefel manifold $V_k(\RR^d)$. Set
\begin{align}\label{eqn:TBprocess}
Y_t := \Max_{n=1}^\infty U_n V_{A_n^\trans (t) }(\omega_n), \qquad t \in \RR^d.
\end{align}
Then $Y$ is a simple max-stable process on $\RR^d$ with the following properties.

\begin{lemma}\label{lemma:TBmethod}
Let $X$ and $Y$ be simple max-stable processes as given by (\ref{eqn:baseprocess}) and (\ref{eqn:TBprocess}), respectively.
\begin{enumerate}[a)]
\item If $X$ is stationary, then $Y$ is stationary.
\item For any $G \in \orth(d)$ the law of $\{Y_{G(t)}\}_{t \in \RR^d}$ and the law of $Y$ coincide.
\item Let $X$ be stationary. The (radial) TCF $\chi^{(Y)}$ of the stationary isotropic process $Y$ can be expressed in terms of the TCF $\chi^{(X)}$ of $X$ by 
\begin{align*}
\chi^{(Y)} = \tb^d_k(\chi^{(X)}).
\end{align*}
\end{enumerate}
\end{lemma}

\begin{proposition}\label{prop:ctsTB}
If $\chi$ is a continuous TCF on $\RR^k$, then $\tb_k^d(\chi)$ is a continuous TCF on $\RR^d$.
\end{proposition}

\begin{remark}
Contrary to correlation functions, not all radially symmetric continuous TCFs on $\RR^d$ arise as $\tb_k^d(\chi)$ for some TCF $\chi$ on $\RR^k$. 
As a counterexample consider the identity 
\begin{align*}
\exp(-t)=\tb_1^3(f)(t) \qquad \text{ with } \qquad f(t)=\frac{\D}{\D t}\left(t \exp(-t) \right)=(1-t)\exp(-t)
\end{align*} 
(\cf \cite[(2.22)]{schlather_12}). While the completely monotone function $\exp(-t)$ is a valid radial TCF on $\RR^3$, the function $f$ cannot be a TCF on $\RR$ since $f$ attains negative values.
\end{remark}

\begin{remark}
The turning bands method is compatible with iterations in the following sense: Let  $q \geq d$ and construct a process $Z$ on $\RR^q$ from the spectral representation of $Y$ on $\RR^d$ by 
\begin{align*}
Z_t := \Max_{n=1}^\infty U_n V_{B_n^\trans \circ A_n^\trans (t) }(\omega_n) = \Max_{n=1}^\infty U_n V_{(A_n \circ B_n)^\trans (t) }(\omega_n), \qquad t \in \RR^q,
\end{align*}
where $(U_n,\omega_n,A_n,B_n)$ is a Poisson point process on $\RR_+ \times \Omega \times V_k(\RR^d) \times V_d(\RR^q)$  with  intensity $u^{-2} \D u \, \nu(\D \omega) \, \sigma^d_k(\D A) \, \sigma^q_d(\D B)$. 
Then $Z=\{Z_t\}_{t \in \RR^q}$ has the same law as 
\begin{align*}
\left\{\Max_{n=1}^\infty U_n V_{C_n^\trans (t) }(\omega_n)\right\}_{t \in \RR^q},
\end{align*}
where $(U_n,\omega_n,C_n)$ is a Poisson point process with  intensity $u^{-2} \D u \, \nu(\D \omega) \, \sigma^q_k(\D C)$ (see Lemma~\ref{lemma:Stiefelcomposition}). Thus, the process $Z$ can be constructed directly from the spectral representation of $X$ without involving $Y$ as a step in between. 
\end{remark}

\subsection{Multiplication with the class $T^d_{\MMMr}$}

Let $X$ be a  stochastically continuous max-stable process on $\RR^d$ with spectral representation as in (\ref{eqn:baseprocess}) with $k=d$ and let $\{B(t)\}_{t \in \RR^d}$ be a measurable process on $\RR^d$ taking values in $\{0,1\}$. We denote the probability space corresponding to $B$ by $(\Omega_B,{\mathcal A}_B,\PP_B)$ and expectation \wrt $\PP_B$ by $\EE_B$.
Further, we require that
\begin{align*}
c_B:=\int_{\RR^d} B(t) \,\D t \in (0,\infty)   
\end{align*}
holds $\PP_B$-almost surely.
Based on these two processes $X$ and $B$ we define another simple max-stable process $Y$ on $\RR^d$ by
\begin{align}\label{eqn:multprocess}
Y_t := \Max_{n=1}^\infty U_n \frac{B_n(t-z_n)}{c_{B_n}} V_{t}(\omega_n), \qquad t \in \RR^d,
\end{align}
where $(U_n,\omega_n,z_n,B_n)$ is a Poisson point process on $\RR_+ \times \Omega \times \RR^d \times \Omega_B$  with intensity $u^{-2} \D u \times \nu \times \nu_d \times \PP_B $.

\begin{lemma}\label{lemma:multmethod}
Let $X$ and $Y$ be simple max-stable processes as given by (\ref{eqn:baseprocess}) for $k=d$ and (\ref{eqn:multprocess}) respectively.
\begin{enumerate}[a)]
\item If $X$ is stationary, then $Y$ is stationary.
\item Let $X$ be stationary. The TCF $\chi^{(Y)}$ of the stationary  process $Y$ can be expressed in terms of the TCF $\chi^{(X)}$ of $X$ by 
\begin{align*}
\chi^{(Y)}(t)= \EE_B\left[\frac{\int_{\RR^d} B(z) B(z-t) \D z}{\int_{\RR^d} B(z) \D z}\right]  \, \chi^{(X)}(t), \qquad t \in \RR^d.
\end{align*}
\end{enumerate}
\end{lemma}

\begin{example}\label{example:MultGneiting}
If the process $B$ on $\RR^d$ is chosen to be the indicator function 
$B(t):=\eins_{\lVert t \rVert \leq R}$ of the ball $B^d_{R}$ for a random radius  $R \in (0,\infty)$, then the function
\begin{align*}
t \mapsto \EE_B\left[\frac{\int_{\RR^d} B(z) B(z-t) \D z}{\int_{\RR^d} B(z) \D z}\right]  = \EE_R \int_{\RR^d}  \frac{\eins_{\lVert z\rVert \leq R} \wedge \eins_{\lVert z-t\rVert \leq R}}{\kappa_d R^d}\,\D z
\end{align*}
depends on $\lVert t \rVert$ only and belongs to the class $T^d_{\MMMr}=T^d_{\MMMb}$ corresponding to the random radius $R$ of an M3b process (cf.\ Table~\ref{table:semiparameterTCFs}).
\end{example}

\begin{remark}
From \cite[Proposition~3.3.1 c)]{strokorb_13} it is already known that multiplication of TCFs on some space yields again TCFs on the same space. The advantage here is the explicit construction of a max-stable process from a given spectral representation. Note that Lemma~\ref{lemma:multmethod} generalizes a construction described in \cite[\page 39]{schlather_02}. 
\end{remark}

\subsection{Examples}\label{sec:cexamples}


Let us denote  
\begin{align*}
T^d_r := \left\{ \chi: [0,\infty) \rightarrow [0,1] \,\, \middle| \,\, \begin{array}{l} \text{$\chi$ continuous radial TCF on $\RR^d$ that is}  \\ \text{convex in the radius and vanishes at $\infty$} \end{array} \right\} = T^d \cap T^1_{\MMMr}.
\end{align*} 
First, we provide {for each $d\geq 3$} an example of a TCF $\varphi_d \in T^d_r \setminus T^d_{\MMMr}$. To this end, consider the \emph{tent function}
\begin{align*}
h_1(t)=(1-t)_+ \qquad  t \geq 0,
\end{align*}
which belongs to $T^1_{\MMMr}$ (\cf (\ref{eqn:ToneMMMr})).
If we apply the turning bands operator, we obtain 
\begin{align}\label{eqn:phid}
\varphi_d(t):=\tb_1^d(h_1)(t) \qquad t\geq 0,
\end{align}
which is a radial TCF on $\RR^d$ (\cf Proposition~\ref{prop:ctsTB}). 

\begin{proposition}\label{prop:TBofTent}
\begin{enumerate}[a)]
\item For $d\geq 1$ we have $\varphi_d \in T^d_r = T^d \cap T^1_{\MMMr}$. 
\item For $d\geq 1$ and $k\geq 3$ we have $\varphi_d \not \in T^k_{\MMMr}$.
\item For $d=1$ and $d\geq 6$ we have $\varphi_d \not \in T^2_{\MMMr}$.
\end{enumerate}
\end{proposition}

\begin{remark}\label{remark:notconvex}
In the remaining cases $d\in\{2,3,4,5\}$  plots of the function $c(t)$ from (\ref{eqn:notconvex}) suggest that $\varphi_d$ does not belong to $T^2_{\MMMr}$ either since $c(t)$ is not convex, see Figure~\ref{fig:notconvex} in Section~\ref{sec:proofs}. 
\end{remark}


\begin{remark}\label{remark:linearTCFsonBalls}
The TCF $\varphi_d$ decreases linearly on the interval $[0,1]$  (\cf (\ref{eqn:phidDerivative}))
\begin{align}\label{eqn:betad}
  \varphi_d(t)=1-\beta_d\, t, \qquad  t \in [0,1], \qquad \text{where} \qquad \beta_d=\frac{\Gamma(d/2)}{\sqrt{\pi}\,\Gamma((d+1)/2)}.
\end{align}
Therefore, the radial function $\chi_\beta(t):= 1 - \beta  t $ is an admissible radial TCF on the $d$-dimensional ball $B^d_r$ of radius $r$ if $\beta \in [0,\beta_d/r]$. This complements results in \cite{gneiting_99a}, where it is shown that $\varphi(t) = 1-\alpha t$ is positive definite on $B^d_r$ if and only if $\alpha \in [0,2\beta_d/r]$. It seems likely that the bound $\beta_d/r$ is sharp for $\chi_\beta(t)$ to be a TCF on $B^d_r$.
\end{remark}

Secondly, combining the turning bands operator and the multiplication operation
leads to an example of a TCF  $\chi_3 \in T^3_r$  that is not contained in any of the classes given in Figure~\ref{fig:M3_BR} for $d=3$, and we suppose  that our example $\chi_d$ satisfies this property also for any other dimension $d\geq 2$. Consider the function 
\begin{align}\label{eqn:chid}
\chi_d(t):=\varphi_d(2t) \, h_d(t), \qquad t\geq 0,
\end{align}
where $\varphi_d$ is from (\ref{eqn:phid}).

\begin{proposition}\label{prop:TBofTentandMult}
\begin{enumerate}[a)]
\item For $d\geq 1$ we have $\chi_d \in T^d_r \setminus T^d_\VBR$.
\item \label{item:notinFigure} For $d=3$ we have $\chi_d \in T^d_r \setminus (T^d_\VBR \cup T^d_\MMMr)$. 
\end{enumerate}
\end{proposition}

\section{Parametric families}\label{sec:parametric}

The considerations above also lead to sharp bounds for some well-known parametric families of positive definite functions to be a TCF, see Table~\ref{table:sharpbounds}. 

The first three families (\emph{powered exponential}, \emph{Whittle-Mat{\'e}rn}, \emph{Cauchy} are completely monotone for the parameters given in Table~\ref{table:sharpbounds} (\cf \cite[(1.2),(1.6) and (2.32)]{millersamko_01} for example), and thus they can be realized by either an MPS process, an M3 process of non-increasing shapes (e.g.\ M2r or M3b) or by a VBR process (in all cases in any dimension).
The \emph{powered error function} is not completely monotone but a member of the class $T^\infty_{\MMMr}$. That means it can be realized by an M3 process of non-increasing shapes or by a VBR process (both in any dimension), but not by an MPS process.
In all of these cases, we may exclude bigger parameters $\nu$ because the {(right-hand) derivative} at $0$ vanishes for bigger $\nu$, but the triangle inequality $\eta(s\pm t)=\eta(s)+\eta(t)$ for $\eta:=1-\chi$ enforces this {derivative} to be negative in order to be a TCF (\cf \cite[Corollary~2]{markov_95} or \cite[Theorem~3~(ii)]{schlathertawn_03}).

The \emph{truncated power function} is an example of a TCF with compact support. Because a TCF has to be positive definite, this leads to the situation that the valid model parameter depends on the dimension. It is chosen such that the function belongs to $T^d_{\MMMr}$ (\cf \cite[Theorem~6.3]{gneiting_99}), and thus can be realized by an M3 process of non-increasing shapes on $\RR^d$. Because of its compact support the function cannot belong to any of the other classes presented in Figure~\ref{fig:M3_BR}.
The bound is sharp in odd dimensions because the function is not positive definite otherwise (\cf \cite[Theorem~1 and \page 165]{golubov_81}). 
For even dimensions this choice is valid but possibly not sharp. Again due to \cite{golubov_81}, we know at least that $\nu$ has to satisfy $\nu \geq (d+1)/2$ in order to be positive definite.

\begin{table}\small
\centering
\begin{tabular}{p{.3\columnwidth}  >{$\displaystyle}p{.27\columnwidth}<{$}  >{$\displaystyle}p{.16\columnwidth}<{$} >{$\displaystyle}p{.15\columnwidth}<{$}}
\toprule \\[-4mm]
\multicolumn{2}{l}{\textbf{{Parametric family of cts. radial functions on $\RR^d$}}}    & \text{\textbf{CF for}} & \text{\textbf{TCF for}} \\[1mm] 
\toprule \\[-1mm]
powered exponential &   \exp(-r^\nu)  & 0 < \nu \leq 2 & 0 < \nu \leq 1 \\[5mm]
Whittle-Mat{\'e}rn & 2^{1-\nu} \, \Gamma(\nu)^{-1}\, r^{\nu}\,K_\nu(r)  & 0 < \nu  & 0 < \nu \leq 0.5 \\[5mm]
Cauchy  & (1+ r^\nu)^{-\beta} \quad \beta > 0 &  0 < \nu \leq 2 &  0 < \nu \leq 1 \\[5mm]
powered error function$^{*}$ & \text{erfc}(r^\nu) &  0 < \nu \leq 1 & 0 < \nu \leq 1 \\[5mm]
truncated power function$^{*}$ & (1-r)_+^\nu  & \nu \geq  (d+1)/2  &  \nu \geq  \lfloor d/2 \rfloor + 1  \\[1mm]
\bottomrule
\end{tabular}
\caption{Parametric families of continuous radially symmetric functions on $\RR^d$ and their sharp parameter bounds for being a correlation function (CF) and for being a tail correlation function (TCF). $^{*}${}The CF bound for the powered error function is sharp if we require validity of the model for all dimensions and the TCF bound for the truncated power function is sharp for odd dimensions.}\label{table:sharpbounds}
\end{table}

\section{Proofs}\label{sec:proofs}

The class $H_d$ in \cite{gneiting_99} is defined as the class of functions $\varphi$ on $[0,\infty)$ of the form 
\begin{align}\label{eqn:Hd}
\varphi(t)=\int_{(0,\infty)} h_d(st) \,\D G(s),
\end{align}
where $G$ is a distribution function on $(0,\infty)$ (with $G(0+)=0$) and where
\begin{align}\label{eqn:functionHd}
h_d(t)= \frac{d \, \Gamma(d/2)}{\sqrt{\pi} \, \Gamma((d+1)/2)} \, \int_t^1 (1-v^2)^{(d-1)/2}_+ \D v.
\end{align}
Here the function $h_d(t)=\tilde{h}_d(t)/\tilde{h}_d(0)$ with $\tilde{h}_d$ is the self-convolution of the ball indicator function $\eins_{B^d_{0.5}}$ viewed as a radial function. It is shown already in \cite{gneiting_99} that $H_d$ and the \emph{Mittal-Berman class} $V_d$ coincide (for $d \geq 2$; \cf \cite[(40)]{gneiting_99}  and \cite{mittal_76}). Here $V_d$ is the class of functions $\varphi$ on $[0,\infty)$ of the form 
\begin{align}\label{eqn:Vd}
\varphi(t)=2 \int_{t/2}^\infty \frac{S_{d,u,\theta(t,u)}}{S_{d,u,\pi}} \, p(u) \,\D u,
\end{align}
where $p$ is a probability density function on $(0,\infty)$, such that $p(u)/u^{d-1}$ is non-increasing, and $S_{d,u,\theta}$ is the surface area of the sphere $\{x \pmid \lVert x \rVert = u\} \subset \RR^d$  intersected by the cone of angle $\theta(t,u)=\arccos(t/(2u))$  (with apex the origin). In what follows, we show that we have 
\begin{align}\label{eqn:Hdcoincide}
H_d = T^d_{\MMMr}=T^d_{\MMr}=T^d_{\MMMb} (=V_d) \qquad \text{for $d\geq 1$ ($d\geq 2$).}
\end{align} 

\begin{proof}[Proof of Proposition~\ref{prop:M3coincide}] 
We divide the proof into five steps:
\begin{description}
 
\item[\normalfont \emph{1st step}] $H_d=T^d_{\MMMb}$ for $d \geq 1$.\medskip\\
By definition, members of the class $T^d_{\MMMb}$ have the form 
\begin{align*}
\chi(\lVert t \rVert)
=\EE_R \left(\frac{1}{\vol_d(B^d_R)} \int_{\RR^d}  \eins_{\lVert z\rVert \leq R} \wedge \eins_{\lVert z-t\rVert \leq R} \,\D z \right)
= \EE_R \left( h_d\left(\frac{\lVert t \rVert}{2R}\right)\right)
\end{align*}
for some random radius $R \in (0,\infty)$. The last equality holds because the integral with the minimum $\wedge$ is in fact a convolution for indicator functions. Therefore, the transformation $S:=1/(2R)$ shows that this $\chi$ and $\varphi$ from (\ref{eqn:Hd}) are equal, when $G$ denotes the law of $S$ on $(0,\infty)$ and vice versa. Hence $T^d_{\MMMb}=H_d$ for $d\geq 1$.

\item[\normalfont \emph{2nd step}] $T^d_{\MMr} = V_d = H_d$ for $d \geq 2$ and (\ref{eqn:fandH}) holds for $d \geq 2$.\medskip\\
Members of $T^d_{\MMr}$ depend on a shape function $f\geq 0$ with $\int_{\RR^d} f(\lVert t \rVert)  \D t = 1$, which is non-increasing as the radius grows, whereas members of $V_d$ depend on a probability density function $p$ on $(0,\infty)$ with $p(u)/u^{d-1}$ non-increasing in $u>0$. Integration along the radius shows that both functions are in one-to-one corresponcence via
\begin{align*}
f(\lVert t \rVert)=\frac{p(\lVert t \rVert)}{S_{d,\lVert t \rVert,\pi}}.
\end{align*}
Moreover, since $f$ is non-increasing, this correspondence is compatible with the integration in  (\ref{eqn:Vd})  and the TCF for M2r processes in Table~\ref{table:semiparameterTCFs}. Hence $T^d_{\MMr} = V_d$ for $d \geq 2$. From \cite{gneiting_99} we already know that $H_d=V_d$. In particular, $f$ and $G$ as in (\ref{eqn:Hd}) can be recovered from each other by (44) and (45) in \cite{gneiting_99} with $n \geq 2$ or, equivalently, $f$ and $H(s)=G(s/2)$ can be recovered from each other by (\ref{eqn:fandH}) with $d \geq 2$ here. Note that our $f$ corresponds to $g$ in \cite{gneiting_99}. 

\item[\normalfont \emph{3rd step}] $T^1_{\MMr} = H_1$ and (\ref{eqn:fandH}) holds for $d=1$.\medskip\\
If $d=1$, it is straightforward to check that for $\chi \in T^1_{\MMr}$ depending on a single shape function $f$, we have 
\begin{align}\label{eqn:MoneHone}
\chi(t)= \int_{\RR} f(z) \wedge f(z-t) \,\D z =2\int_{t/2}^\infty f(u) \,\D u
\end{align}
(similarly to the integration along the radius in (\ref{eqn:Vd})). Now, precisely the same proof as the proof of Theorem~5.2. in \cite{gneiting_99} applies here when we set $n=1$, $g=f$, $\varphi=\chi$ and omit the term $S_{n,u,\theta}$ in (48) and (49) therein, showing that $T^1_{\MMr}=H_1$. 
In particular, $f$ and $G$ as in (\ref{eqn:Hd}) can also be recovered from each other by  (44) and (45) in \cite{gneiting_99} with $n=1$ or, equivalently, $f$ and $H(s)=G(s/2)$ can be recovered from each other by (\ref{eqn:fandH}) with $d=1$ here (where our $f$ corresponds to $g$ therein). 

\item[\normalfont \emph{4th step}] $T^d_{\MMMr} \subset H_d$ for $d \geq 1$.\medskip\\
From the 2nd and 3rd step we know that $T^d_{\MMr}=H_d$ for $d\geq 1$. That means for each (single deterministic) radially symmetric non-increasing shape function $f \geq 0$ on $\RR^d$ with $0 < \lVert f \rVert_{\text{L}^1(\RR^d)} < \infty$ we may define a unique distribution function $G_{f/\lVert f \rVert_{\text{L}^1(\RR^d)}} (s)= H_{f/\lVert f \rVert_{\text{L}^1(\RR^d)}} (2s)$ via (\ref{eqn:fandH}). We set 
\begin{align*} 
A(f)_s:= \lVert f \rVert_{\text{L}^1(\RR^d)} \,\, G_{f/\lVert f \rVert_{\text{L}^1(\RR^d)}}(s) \qquad s>0
\end{align*}
such that $A(f)$ is non-decreasing on $(0,\infty)$ with $A(f)_{0+}=0$, right-continuous and $A(f)$ has total variation $\lVert f \rVert_{\text{L}^1(\RR^d)}$. It is coherent to set $A(0)\equiv 0$. 
Now, consider a member $\chi$ of $T^d_{\MMMr}$ and its corresponding measurable process $\left\{f(t)\right\}_{t \in \RR^d}$, which satisfies $\EE_f(\lVert f \rVert_{\text{L}^1(\RR^d)})=1$. Then $\left\{A(f)_s\right\}_{s>0}$ defines a non-decreasing, right-continuous process with $\EE \left(A(f)_\infty\right) = 1$ and $A(f)_{0+}=0$. 
Moreover, note that (by the correspondence $T^d_{\MMr}=H_d$)
\begin{align*}
\chi(t)=\EE_f\left(\int_0^\infty h_d(st) \,\D A(f)_s\right).
\end{align*}
Set $G(s):=\EE_f A(f)_{s}$. Then $G$ is also non-decreasing, right-continuous with total variation 1 and with $G(0+)=0$ (by dominated convergence).
Finally, we obtain (again by dominated convergence) that
\begin{align*}
\chi(t)=\int_0^\infty h_d(st) \,\D \EE_fA(f)_s = \int_0^\infty h_d(st) \,\D G(s)
\end{align*}
as desired. Hence $T^d_{\MMMr} \subset H_d$.

\item[\normalfont \emph{5th step}] (Summary) From the previous steps we know that $T^d_{\MMMr} \subset H_d = T^d_{\MMMb} =T^d_{\MMr}$ for $d\geq 1$. Clearly, $T^d_{\MMMb} \subset T^d_{\MMMr}$ by definition, so that $T^d_{\MMMr}$,$T^d_{\MMr}$,$T^d_{\MMMb}$,$H_d$ coincide for $d\geq 1$. \qedhere
 
\end{description}

\end{proof}

\begin{proof}[Proof of Proposition~\ref{prop:MixingInclusions}]
\begin{enumerate}[a)]
\item If $\varphi$ is completely monotone, then also $-\varphi'$ and $-\varphi'(\sqrt{\cdot})$ will be completely monotone, since $\sqrt{\cdot}$ is a Bernstein function. This shows $T^d_{\MPS} \subset T^\infty_{\MMMr}$. Clearly, $T^\infty_{\MMMr} \subset T^d_{\MMMr}$.
\item Clearly, $T^d_{\BR} \subset T^\infty_{\VBR}$, since BR process form a proper subclass of VBR processes. The inclusion $T^d_{\VBR} \subset T^\infty_{\MMMr}$ follows from (\ref{eqn:TdVBR}), since $\gamma(t)=8\lVert t \rVert^2$ is a valid variogram in each dimension.
\item The variogram $\gamma(t)=8 \lVert t \rVert^{2\alpha}$ is valid in each dimension for $\alpha \in (0,1]$ (corresponding to fractal Brownian motion). Hence $\erfc(t^\alpha)$ is a valid TCF of a BR process for $\alpha \in (0,1]$. Moreover, the function $\erfc(t^\alpha)$ is completely monotone if and only if  $\alpha \leq 0.5$. \qedhere
\end{enumerate}
\end{proof}


\begin{proposition}\label{prop:absmon}
Let $R$, $S_\lambda$ and $T_\lambda=R \circ S_\lambda$ be the maps from (\ref{eqn:R}),(\ref{eqn:Slambda}) and (\ref{eqn:Tlambda}) and set $R_\alpha(x):=R((1-\alpha)x+\alpha)$, $S_{\lambda,\alpha}(x):=S_\lambda((1-\alpha)x+\alpha)$ and $T_{\lambda,\alpha}(x):=T_\lambda((1-\alpha)x+\alpha)$. Then $R_\alpha$, $S_{\lambda,\alpha}$ and $T_{\lambda,\alpha}$  are continuous on $[-1,1]$ and analytic on $(-1,1)$ for all $\lambda > 0$ and $\alpha \in [0,1]$. 
\begin{enumerate}[a)]
\item\label{itemi:absmon} The function $R_\alpha$ is absolutely monotone on $[0,1]$ for $\alpha \geq 0.5$.
\item\label{itemii:absmon} The function $S_{\lambda,\alpha}$ is absolutely monotone on $[0,1]$ for $\lambda (1-\alpha) \leq 8 (\erf^{-1}(1/\sqrt{2}))^2$. 
\item\label{itemiii:absmon} The function $T_{\lambda,\alpha}=R \circ S_{\lambda,\alpha}$  is absolutely monotone on $[0,1]$ for $\lambda(1-\alpha) \leq 8 (\erf^{-1}(1/2))^2$. 
\end{enumerate}
\end{proposition}

To deal with the function $S_{\lambda,\alpha}$ in Proposition~\ref{prop:absmon}, we first prove an auxiliary lemma, which might be interesting in its own right.

\begin{lemma}\label{lemma:cmfunction}
The function $f(x)=1-\left(\erf(\sqrt{x})\right)^2$ is completely monotone on $[0,\infty)$.
\end{lemma}
\begin{proof} The function $f$ is non-negative, continuous on $[0,\infty)$ and the first derivative of $f$ on $(0,\infty)$ is given by
\begin{align*}
f'(x)=- \frac{2}{\sqrt{\pi}}\frac{\erf(\sqrt{x})}{\sqrt{x}} {e^{-x}} \qquad x>0.
\end{align*}
Now, the functions $e^{-x}$ and $\erf{\sqrt{x}}/\sqrt{x}$ are completely monotone on $(0,\infty)$ (cf. \cite[(1.2) and Corollary to Theorem~5]{millersamko_01}). Hence, $-f'$ is completely monotone, which shows that $f$ is completely monotone on $[0,\infty)$.
\end{proof}


\begin{proof}[Proof of Proposition~\ref{prop:absmon}]
It can be seen directly that the functions $R_\alpha$, $S_{\lambda,\alpha}$, $T_{\lambda,\alpha}$ 
are continuous on $[-1,1]$ and analytic on $(-1,1)$ for all $\lambda > 0$. 
\begin{enumerate}[a)]
\item Using the series expansion of the cosine function, we arrive at
\begin{align*}
&R(x)
=\sum_{n=0}^\infty \frac{(-1)^n}{(2n)!} \pi^{2n} \frac{(1-x)^n}{2^n}
=\sum_{n=0}^\infty \frac{\pi^{2n}}{2^n(2n)!} \sum_{k=0}^n \binom{n}{k} x^k(-1)^{n-k}\\
&=\sum_{k=0}^\infty x^k \sum_{n=0}^\infty  \frac{\pi^{2n+2k}}{2^{n+k}(2n+2k)!} \binom{n+k}{k} (-1)^{n}\\
&=R(0)+\sum_{k=1}^\infty x^k \frac{\pi^{2k}}{2^{2k}k!} \sum_{n=0}^\infty (-1)^n \frac{\pi^{2n}}{2^n(2n)!} \frac{1}{(2n+2k-1)(2n+2k-3) \dots (2n+1)}.
\end{align*}
We show that the coefficients 
\begin{align*}
a_k:=\sum_{n=0}^\infty (-1)^n \frac{\pi^{2n}}{2^n(2n)!} \frac{1}{(2n+2k-1)(2n-2k-3) \cdots (2n+1)}
\end{align*}
are non-negative for $k\geq 1$: Since this series representing $a_k$ converges absolutely, we may partition by even ($n=2\ell$) and odd ($n=2\ell+1$) coefficients:
\begin{align*}
&a_k =\sum_{\ell=0}^\infty \frac{\pi^{4\ell}}{2^{2\ell}(4\ell)!} \frac{1}{(4\ell+2k-1) \cdots (4\ell+1)}\\
&\phantom{a_k = \,} - 
\sum_{\ell=0}^\infty \,
\frac{\pi^{4\ell+2}}{2^{2\ell+1}(4\ell+2)!} \frac{1}{(4\ell+2k+1) \cdots (4\ell+3)}\\
&=\sum_{\ell=0}^\infty \, \frac{\pi^{4\ell}}{2^{2\ell}(4\ell)!} \frac{1}{(4\ell+2k-1) \cdots (4\ell+3)}
\left[ 
\frac{1}{4\ell+1}-\frac{{\pi^2}/{2}}{(4\ell+2)(4\ell+1)}\frac{1}{4\ell+2k+1}
\right] 
\end{align*} 
Now, the expression in the brackets is positive since $k\geq 1$ and $\ell \geq 0$. 
Thus, $a_k > 0$ for $k \geq 1$. In particular, $R(x)-R(0)$ and $R_\alpha(x)-R_\alpha(0)$ are absolutely monotone on $[0,1]$. If $\alpha \geq 0.5$, then $R_\alpha(0)\geq 0$.
\item
Lemma~\ref{lemma:cmfunction} tells us that $f(x)=1-\left(\erf(\sqrt{x})\right)^2$ is completely monotone on $[0,\infty)$. Now, $S_\lambda(x)= 2 f\left(\lambda(1-x)/8\right)-1$. Hence, the $k$-th derivative for $k\geq 1$ satisfies 
\begin{align*}
S^{(k)}_\lambda(x)=2 \left(\frac{\lambda}{8}\right)^k (-1)^k f^{(k)}\left(\frac{\lambda}{8}(1-x)\right) \geq 0.
\end{align*}
In particular, all but eventually the $0$-th Taylor coefficient $S_\lambda(0)$ are non-negative, and  $S_\lambda(0)$ is non-negative if and only if $\lambda \leq 8 (\erf^{-1}(1/\sqrt{2}))^2$. Note that $S_{\lambda,\alpha} = S_{\lambda (1-\alpha)}$.
\item Since $T_\lambda=R \circ S_\lambda$ and $T'_\lambda=(R'\circ S_\lambda)  S'_\lambda$, it follows from the proof of \ref{itemi:absmon}) and \ref{itemii:absmon}) that all but eventually the $0$-th Taylor coefficient $T_\lambda(0)$ are non-negative, and  $T_\lambda(0)$ is non-negative if and only if $\lambda \leq 8 (\erf^{-1}(1/2))^2$. Note that $T_{\lambda,\alpha} = T_{\lambda (1-\alpha)}$. \qedhere
\end{enumerate}
\end{proof}

\begin{proof}[Proof of Proposition~\ref{prop:cortrafo}]
Since convex combinations, products and (pointwise) limits of correlation functions are again correlation functions,  a map $B:[-1,1] \rightarrow [-1,1]$ transforms correlation functions again into correlation functions if $B$ is continuous on $[-1,1]$ and analytic on $(-1,1)$, such that the respective Taylor series at $0$ has only non-negative coefficients.  Such functions are {absolutely monotone} on $[0,1]$ and conversely, the Taylor series representation at $0$ of an absolutely monotone function on $[0,1]$ extends to $[-1,1]$. So the assertion follows from Proposition~\ref{prop:absmon} with $B \in \{R_\alpha,S_{\lambda,\alpha},T_{\lambda,\alpha}\}$.
\end{proof}

\begin{proof}[Proof of Proposition~\ref{prop:MixingNonempty}]
\begin{enumerate}[a)]
\item Cf.\ the proof of Proposition~\ref{prop:MixingInclusions}\ref{item:erfcCondition}) and note that the function $\erfc(t^\alpha)$ belongs to $T^\infty_\MMMr$ if and only if $\alpha \in (0,1]$.
\item
The class $T^d_\MMMr=H_d$ naturally contains functions with compact support, \eg the function $h_d$ (\cf (\ref{eqn:Hd})), whereas $T^d_\VBR$ cannot contain such functions.  To see this, recall (\ref{eqn:TinftyMMMr}) and observe that members of $T^\infty_\MMMr$ are scale mixtures of $\erfc$ that cannot have compact support. Thus, the involved variogram in (\ref{eqn:TdVBR}) would have to take the value $\infty$ outside a compact region. 
\item
Consider the simple $\erfc$-mixture  
\begin{align*}
\chi(\lVert t \rVert)  = 0.25 \cdot \erfc (\lVert t \rVert) + 0.75 \cdot \erfc(5 \lVert t \rVert) \qquad t \in \RR^d.
\end{align*} 
Surely, $\chi$ is a member of $T^\infty_\MMMr$ (\cf (\ref{eqn:TinftyMMMr})).  Suppose that there is a BR process on $\RR^d$  corresponding to a variogram $\tilde \gamma$ such that its TCF $\tilde \chi$ coincides with $\chi$. We will show now that this cannot be true for any dimension $d$. Otherwise, 
\begin{align*}
\tilde \gamma(\lVert t \rVert)= 8 \left[\erfc^{-1} \left(0.25 \cdot \erfc (\lVert t \Vert) + 0.75 \cdot \erfc(5 \lVert t \rVert)\right)\right]^2 \qquad t \in \RR^d
\end{align*} 
is a variogram for any dimension $d$.  In particular, $\tilde \gamma(\lVert \cdot \rVert)$ is for any dimension $d$ a continuous negative definite function on $\RR^d$. By \cite[5.1.8]{bcr_84} it follows that the function 
\begin{align*}
\psi(r) = \left[\erfc^{-1} \left(0.25 \cdot \erfc (\sqrt{r}) + 0.75 \cdot \erfc(5 \sqrt{r})\right)\right]^2 \qquad r \in [0,\infty) 
\end{align*}
is a (continuous) negative definite function on $[0,\infty)$ in the semigroup sense and obviously $\psi(r)\geq 0$. Hence $\psi(r)$ is a Bernstein function (cf. \cite[4.4.3]{bcr_84}). However, the second derivative of $\psi(r)$ has a local minimum. So, the assertion fails and our assumption must be wrong. That means there is a dimension $d_0$ such that the above $\chi \in T^\infty_\MMMr$ cannot be realized as a TCF of a BR process for any dimension $d \geq d_0$. \qedhere
\end{enumerate}
\end{proof}

\begin{lemma} \label{lemma:Stiefelcomposition}
Let $k_1 \leq k_2 \leq k_3$.
\begin{enumerate}[a)]
\item The composition map 
\begin{align*}
V_{k_1}(\RR^{k_2}) \times V_{k_2}(\RR^{k_3}) \rightarrow V_{k_1}(\RR^{k_3}) \qquad (A,B) \mapsto B \circ A
\end{align*} 
is continuous.
\item If $B \sim \sigma^{k_3}_{k_2}$ is uniformly distributed on $V_{k_2}(\RR^{k_3})$ and $A$ is an independent (Borel-measurable) random variable with values in $V_{k_1}(\RR^{k_2})$, then the composition $B \circ A$ will also be uniformly distributed $B \circ A \sim \sigma^{k_3}_{k_1}$.
\item The turning bands operator is compatible with compositions
\begin{align}
\tb_{k_1}^{k_2} \circ \tb_{k_2}^{k_3} = \tb_{k_1}^{k_3}.
\end{align}
\end{enumerate}
\end{lemma}

\begin{proof}[Proof of Lemma~\ref{lemma:Stiefelcomposition}]
\begin{enumerate}[a)]
\item The composition of matrices is continuous and here just restricted to a subspace.
\item \label{item:Stiefelcomposition} Let $f$ be a continuous function on $V_{k_1}(\RR^{k_3})$, then (by dominated convergence) the function $g(b):=\EE_A(f(b\circ A))$ will also be continuous on $V_{k_2}(\RR^{k_3})$. 
Therefore, $\EE_B(g(G^{-1} B))= \EE_B(g(B))$ for all $G \in \orth(k_3)$, since $B \sim \sigma^{k_3}_{k_2}$. Thus, we also have  for $G \in \orth(k_3)$ that
\begin{align*}
\EE f (G^{-1} \circ B \circ A) &= \EE ( \EE ( f ( G^{-1} \circ B \circ A ) | B ) )
= \EE (g(G^{-1} B))  \\
&= \EE (g(B)) = \EE ( \EE ( f ( B \circ A ) | B ) )  = \EE f (B \circ A).
\end{align*}
\item The assertion follows from part \ref{item:Stiefelcomposition}). \qedhere
\end{enumerate}
\end{proof}

\begin{proof}[Proof of Lemma~\ref{lemma:TBmethod}]
Let $M$ be a non-empty finite subset of $\RR^d$ and  $x \in (0,\infty)^M$. The finite-dimensional distributions of $Y$ are determined by 
\begin{align*}
- \log \PP(Y_{t} \leq x_t, \, t \in M) = \int_{V_k(\RR^d)} \int_{\Omega} \left(\Max_{t \in M}\frac{V_{A^\trans t}(\omega)}{x_t}\right) \,\nu (\D \omega) \, \sigma^d_k(\D A).
\end{align*} 
\begin{enumerate}[a)]
\item  If $X$ is stationary, then
\begin{align*}
\int_{\Omega} \left(\Max_{t \in M}\frac{V_{A^\trans (t+h)}(\omega)}{x_t}\right) \,\nu (\D \omega) = \int_{\Omega} \left(\Max_{t \in M}\frac{V_{A^\trans t}(\omega)}{x_t}\right) \,\nu (\D \omega),
\end{align*}
for all $h \in \RR^d$ and all $A \in V_k(\RR^d)$, since $A$ is linear.
\item This follows since $\sigma^d_k$ is $\orth(d)$-invariant.
\item The assertion follows from (\ref{eqn:chispectral}). \qedhere
\end{enumerate}
\end{proof}

\begin{proof}[Proof of Proposition~\ref{prop:ctsTB}]
In view of Lemma~\ref{lemma:TBmethod}
we need to show that continuous TCFs on $\RR^k$ coincide with the TCFs of stochastically continuous processes on $\RR^k$. Therefore, let $\chi$ be a continuous TCF on $\RR^k$ and let $X$ be a corresponding stationary max-stable process. 
Let $\theta$ be the extremal coefficient function (ECF) of $X$ as in {\cite{strokorbschlather_13}} and let $X^*$ be the associated Tawn-Molchanov process as in {\cite[Theorem~8]{strokorbschlather_13}}. Note that $\chi(h)=2-\theta(\{h,o\})$. By construction, $X^*$ is also stationary and has TCF $\chi$. Additionally, $X^*$ is stochastically continuous due to {\cite[Theorem~25]{strokorbschlather_13}}.
\end{proof}

\begin{proof}[Proof of Lemma~\ref{lemma:multmethod}]
Let $M$ be a non-empty finite subset of $\RR^d$ and  $x \in (0,\infty)^M$. The finite-dimensional distributions of $Y$ are determined by 
\begin{align*}
- \log \PP(Y_{t} \leq x_t, \, t \in M) = \EE_B \int_{\RR^d} \int_{\Omega} \left(\Max_{t \in M}\frac{B(t-z)V_{t}(\omega)}{c_B x_t}\right) \,\nu (\D \omega) \, \D z.
\end{align*} 
\begin{enumerate}[a)]
\item  If $X$ is stationary, then
\begin{align*}
\int_{\Omega} \left(\Max_{t \in M}\frac{B(t-z)V_{t+h}(\omega)}{x_t}\right) \,\nu (\D \omega)
=
\int_{\Omega} \left(\Max_{t \in M}\frac{B(t-z)V_{t}(\omega)}{x_t}\right) \,\nu (\D \omega)
\end{align*}
for all $h \in \RR^d$, all $z \in \RR^d$ and all $B \in \{0,1\}^{\RR^d}$. Therefore,
\begin{align*}
&\int_{\RR^d} \int_{\Omega} \left(\Max_{t \in M}\frac{B((t+h)-z)V_{t+h}(\omega)}{x_t}\right) \,\nu (\D \omega) \, \D z\\
&=
\int_{\RR^d} \int_{\Omega} \left(\Max_{t \in M}\frac{B(t-z)V_{t}(\omega)}{x_t}\right) \,\nu (\D \omega) \, \D z
\end{align*}
for all $h \in \RR^d$ and all integrable functions $B \in \{0,1\}^{\RR^d}$.
\item The assertion follows from (\ref{eqn:chispectral}) and the fact that $b_1 v_1 \wedge b_2 v_2 = b_1b_2 (v_1 \wedge v_2)$ for real numbers $b_1,b_2,v_1,v_2$ with $b_i \in \{0,1\}$ for $i=1,2$. \qedhere
\end{enumerate}
\end{proof}

In the sequel, we shall often write $H_d$ as in \cite{gneiting_99} instead of $T^d_{\MMMr}$, $T^d_{\MMr}$ or $T^d_{\MMMb}$, since all classes coincide (see (\ref{eqn:Hdcoincide})).

\begin{lemma}\label{lemma:TBpreservesHone}
For all $1\leq k \leq d$ The turning bands operator $\tb_k^d$ transfers members of the class $H_1$ into members of $H_1$.
\end{lemma}
\begin{proof}
The class $H_1$ is the class of continuous functions $h$ on $[0,\infty)$ that are convex and satisfy $h(0)=1$ and $\lim_{t \to \infty} h(t)=0$. 
All properties are preserved under $\tb^d_k$. For continuity and $\lim_{t \to \infty} h(t)=0$ use the dominated convergence theorem. Preservation of convexity follows from  $\tb_k^d(h)(r)=\EE_A(h(r c(A)))$ for $r \geq 0$ with $A \sim \sigma^d_k$ and $c(A)=\lVert A^\trans (1,0,\dots,0)^\trans \rVert$.
\end{proof}

\begin{proof}[Proof of Proposition~\ref{prop:TBofTent}]
A priori it is clear that $\varphi_1=h_1$ does not belong to $H_k$ for $k \geq 2$ \cite{gneiting_99}.
\begin{enumerate}[a)]
\item Because of Proposition~\ref{prop:ctsTB} the function $\varphi_d$ is a radial TCF on $\RR^d$. Lemma~\ref{lemma:TBpreservesHone} shows that $\varphi_d=\tb_1^d(h_1)$ belongs to $H_1$. 
\item 
By \cite[equation (6)]{gneiting_99a}  $\varphi_d$ can be expressed as 
\begin{align}\label{eqn:phidexplicit}
\varphi_d(t)=\frac{2\,\Gamma(d/2)}{\sqrt{\pi}\,\Gamma((d-1)/2)} \int_0^1 h_1(tw) (1-w^2)^{(d-3)/2} \D w.
\end{align}
Thus, we have for $d \geq 2$ that
\begin{align}\label{eqn:phidDerivative}
-\varphi'_d\left(\sqrt{t}\right)= \beta_d \left\{ \begin{array}{ll} 1 & \qquad t \leq 1 \\ 1-(1-1/t)^{(d-1)/2} & \qquad t > 1   \end{array}\right.,
\end{align}
where $\beta_d$ is the constant from (\ref{eqn:betad}). Clearly, $-\varphi_d'(\sqrt{t})$ is not convex.
Therefore, one of the conditions of Theorem~3.1 in \cite{gneiting_99} (that is necessary to belong to the class $H_3$) is not fulfilled.
\item
We verify that one of the conditions of Theorem~3.3 in \cite{gneiting_99} (that is  necessary to belong to the class $H_2$) is not fulfilled: Namely, we show that for all $d \geq 6$ the function 
\begin{align}\label{eqn:notconvex}
c(t):= \int_0^t \sqrt{\frac{v}{t-v}} \left(- \varphi'_d\left(1/\sqrt{v}\right)\right) \D v
= \int_0^1 \sqrt{\frac{w}{1-w}} \left(- \varphi'_d\left(1/\sqrt{tw}\right) \cdot t \right) \D w
\end{align}
is not convex. From (\ref{eqn:phidDerivative}) we see that
\begin{align*}
-\varphi'_d\left(1/\sqrt{v}\right)= \beta_d \left\{ \begin{array}{ll} 1-(1-v)^{(d-1)/2} & \qquad v < 1 \\ 1 & \qquad v \geq 1  \end{array}\right.
\end{align*}
Since $d \geq 6$  we can compute the second derivative of $c$ at $1$:
\begin{align*}
  c''(1)& = \int_0^1 \sqrt{\frac{w}{1-w}} \,\cdot \, \left.\frac{\D^2}{\D t^2}\right|_{t=1}\left(- \varphi'_d\left(1/\sqrt{tw}\right) \cdot t \right) \D w\\
& = \beta_d (d-1) \int_0^1 w^{3/2}(1-w)^{(d-6)/2} (1-w\, (d+1)/4) \D w\\ 
& = \beta_d (d-1) \left(B\left(\frac{5}{2},\frac{d-4}{2}\right)-\frac{d+1}{4}\, B\left(\frac{7}{2},\frac{d-4}{2}\right) \right)\\
& = - \beta_d (d-1) \, \frac{3\,\sqrt{\pi}\,\Gamma(d/2-2)}{16\,\Gamma((d+1)/2)}  < 0
\end{align*}
Here $B(x,y)=\int_0^1 t^{(x-1)} (1-t)^{y-1} \D t$ denotes the Beta function. Since $c''(1)$ is negative, the function $c$ cannot be convex. This finishes the proof.
\qedhere
\end{enumerate}
\end{proof}


\begin{figure}\footnotesize
\centering
\scalebox{0.7}{\includegraphics{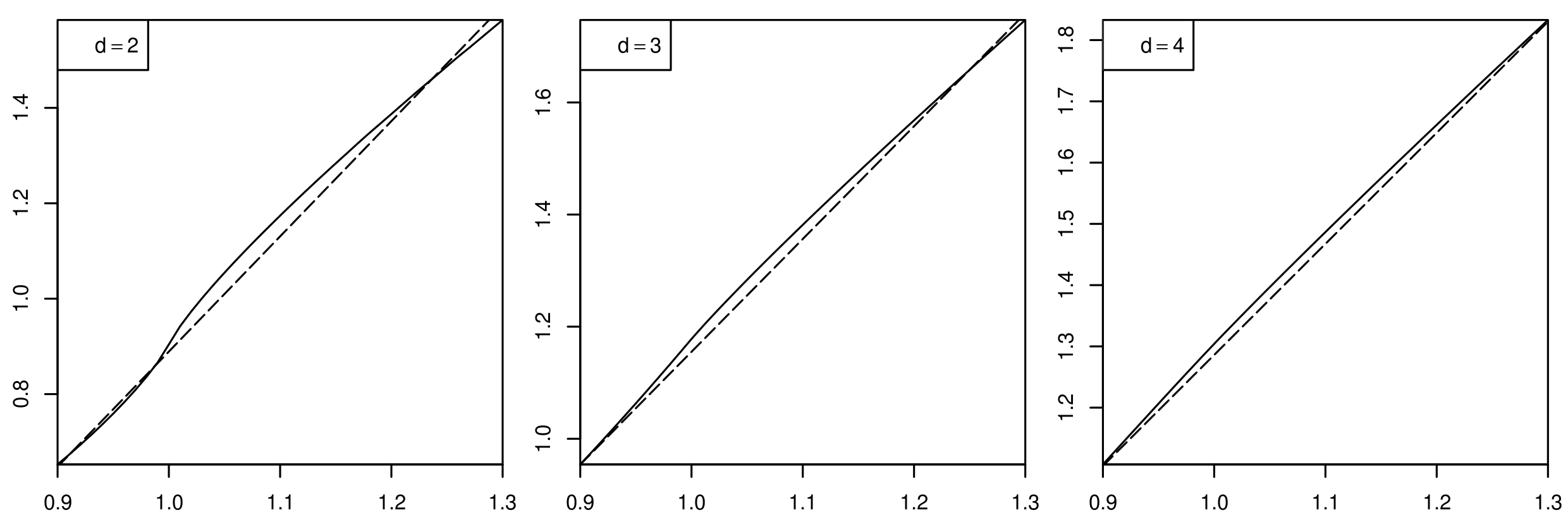}}
\caption{The function $c(t)/\beta_d$ with $c(t)$ from (\ref{eqn:notconvex}) and $\beta_d$ as in (\ref{eqn:betad}) is plotted for $d \in \{2,3,4\}$ (solid line). Additionally the  dashed straight line  indicates that the respective functions are not convex (\cf Remark~\ref{remark:notconvex}). }\label{fig:notconvex}
\end{figure}

\begin{lemma}\label{lemma:HoneMult}
If $f, g\in H_1$ then the product also belongs to this class $fg \in H_1$.
\end{lemma}
\begin{proof}
This is an immediate consequence of \cite[Lemma~4.7]{gneiting_99} (or \cite[Lemma~2]{williamson_56}) which states that if $f$ and $g$ are non-negative, non-increasing and convex on an interval, then the product $f g$ is also non-negative, non-increasing and convex there.
\end{proof}

\begin{proof}[Proof of Proposition~\ref{prop:TBofTentandMult}]
\begin{enumerate}[a)]
\item From Proposition~\ref{prop:TBofTent} we know that $\varphi_d(2t)$ is a radial TCF on $\RR^d$ that belongs to $H_1$. Since $h_d(t)$ belongs to $H_d$ it follows from Example~\ref{example:MultGneiting} that the product $\chi_d(t)=\varphi_d(2t) h_d(t)$ is a radial TCF on $\RR^d$. Moreover $h_d(t)$ also belongs to $H_d \subset H_1$ and therefore $\chi_d \in H_1$ due to Lemma~\ref{lemma:HoneMult}.
  However, $\chi_d \not \in T^d_{\VBR}$ because of its compact support (\cf Propostion \ref{prop:MixingNonempty} \ref{item:M3ohneVBR}).
\item It suffices to show that the function
\begin{align*}
f(t):=-\chi'_3(\sqrt{t})
= -2\varphi'_3(\sqrt{4t})h_3(\sqrt{t})+\varphi_3(\sqrt{4t}) (-h'_3(\sqrt{t}))
\end{align*}
is not convex, because then one of the conditions of Theorem~3.1 in \cite{gneiting_99} (that is necessary to belong to the class $H_3$) is not fulfilled. From (\ref{eqn:functionHd}), (\ref{eqn:phidexplicit}) and (\ref{eqn:phidDerivative}) we see that for $t \in [0,1]$
\begin{align*}
h_3(\sqrt{t})& =\frac{1}{2}(2-3t^{1/2}+t^{3/2}),
\qquad 
& -h'_3(\sqrt{t}) &= \frac{3}{2}(1-t), \\
\varphi_3(\sqrt{4t})& =\left\{ 
\begin{array}{ll} 
1-\sqrt{t} & \quad t \leq 1/4 \\
1/(4\sqrt{t}) & \quad t \geq  1/4 
\end{array}, \right. 
\qquad 
& -2\varphi'_3(\sqrt{4t}) &=\left\{ 
\begin{array}{ll} 
1 & \quad t \leq 1/4 \\
1/(4t) & \quad t \geq  1/4 
\end{array}. \right. 
\end{align*}
Thus, $f(t)$ is a decreasing function on $[0,1]$ with the following left-hand and right-hand derivative at $1/4$
\begin{align*}
\lim_{t \uparrow 1/4}f'(t)=-3 \qquad  \text{ and } \qquad  \lim_{t \downarrow 1/4}f'(t)=-17/4.  
\end{align*}
Hence, $f$ cannot be convex in a neighbourhood of $1/4$.
\qedhere
\end{enumerate}
\end{proof}

\subsection{Derivation of expressions in tables}\label{sec:tableproofs}

\begin{lemma}\label{lemma:recoverGandf_odddim}
Let $d\geq 3$ be odd and $\varphi \in H_d=T^d_{\MMr}$. Let $G$ be a corresponding distribution function as in (\ref{eqn:Hd}) in the definition of the class $H_d$ and let $f$ be a non-increasing shape function as in the definition of the class $T^d_{\MMr}$. Set $k:=(n-1)/2$ and define the (right-hand) derivative
\begin{align*}
\lambda(t):=(-1)^k \frac{\D^k}{\D t^k}\left[-\varphi'\left(\sqrt{t}\right)\right]\qquad t\geq 0
\end{align*}
Then $G$ and $f$ can be recovered from $\varphi$ by
\begin{align*}
G(s)= \frac{\sqrt{\pi}}{d \, \Gamma(d/2)} \int_{0}^s \frac{1}{t^d} \, \D \lambda\left(\frac{1}{t^2}\right) 
\qquad \text{ and } \qquad
f(u)= \left(\frac{2}{\sqrt{\pi}}\right)^{d-1} \lambda\left(4 u^2\right).
\end{align*}
\end{lemma}
\begin{proof}
The recovery of $G$ is precisely \cite[Theorem~3.2]{gneiting_99}. By ({\ref{eqn:fandH}}) with $G(s)=H(2s)$ we obtain
\begin{align*}
f(u)&=\frac{1}{\kappa_d} \int_0^{1/(2u)} (2s)^d \D G(s)
=\frac{\sqrt{2^d \pi}}{\kappa_d \, d \, \Gamma(d/2)}\int_0^{1/(2u)} \, \D \lambda\left(\frac{1}{s^2}\right) \\
&= \left(\frac{2}{\sqrt{\pi}}\right)^{1/(2u)} \left(\lambda(4u^2) - \lim_{x \to \infty} \lambda(x)\right)
\end{align*}
But $\lim_{x \to \infty} \lambda(x)$  necessarily vanishes, since $\lambda(t)=-a'(t)$ for a non-negative (\ie bounded from below), non-increasing and convex function $a(t)$ due to \cite[(22)]{gneiting_99}.
\end{proof}

{ 
\begin{proof}[Proof of Table~\ref{table:recoverfGg}]
Let $G$ denote the distribution function of $1/(2R)$. If the density $g$ of $G$ exists, then the density $k$ of $2R$ is given by $k(s)=g(1/s)/s^2$. In what follows, we show how to recover $G$, its density $g=G'$ and the shape function $f$:\\
In case $d=1$ we refer to \cite[(18)]{gneiting_99} for the recovery of $G$ and $g=G'$. The recovery of $f$ follows from (\ref{eqn:MoneHone}).
In case $d=3$ the previous Lemma~\ref{lemma:recoverGandf_odddim} can be applied to $d=3$ and $\varphi=\chi$, where we abbreviate $\lambda_{\chi}(t)=2\lambda(1/t^2)=t \chi''(1/t)$. 
In case $d=2$ we  additionally assume that $\chi \in H_5$, such that 
\begin{align*}
(-1)^k \frac{\D^k}{\D t^k}\left[-\chi'\left(\sqrt{t}\right)\right]
\end{align*}
exists for $k\in \{0,1,2\}$ and is non-negative, non-increasing and convex for $k\in\{0,1\}$ (\cf \cite[\page 96]{gneiting_99}). This requirement ensures that we can apply the monotone convergence theorem iteratively when differentiating within the following integral (\ref{eqn:muderivative}). A priori we know from \cite[Theorem~3.4]{gneiting_99} that 
\begin{align*}
G(r)=\frac{1}{2} \int_{(0,r)} \frac{1}{s} \,\D \mu(s^2) 
\qquad \text{ with } \qquad
\mu(t)=\frac{\D}{\D t} \int_{0}^t \sqrt{\frac{v}{t-v}} \left[-\chi'(1/\sqrt{v})\right] \, \D v.
\end{align*}
Now $\chi \in H_5$ ensures that $\mu'(t)$ exists by
\begin{align}
\notag \mu'(t) 
&=\frac{\D^2}{\D t^2} \int_{0}^t \sqrt{\frac{v}{t-v}} \left[-\chi'(1/\sqrt{v})\right] \, \D v 
=\frac{\D^2}{\D t^2} \left(t \int_{0}^1 \sqrt{\frac{w}{1-w}} \left[-\chi'(1/\sqrt{wt})\right] \, \D w\right) \\
\label{eqn:muderivative}
&=\int_{0}^1 \sqrt{\frac{w}{1-w}} \left(\frac{\D^2}{\D t^2} \left[- t \, \chi'(1/\sqrt{wt})\right] \right) \, \D w,
\end{align}
where 
\begin{align*}
\frac{\D^2}{\D t^2} \left[- t \, \chi'(1/\sqrt{wt})\right] 
= \frac{1}{4 \, t \, \sqrt{wt}} \left[ \chi''\left(\frac{1}{\sqrt{wt}}\right) - \frac{1}{\sqrt{wt}} \chi'''\left(\frac{1}{\sqrt{wt}}\right) \right].
\end{align*}
The substitutions $v=wt$ and $v=u^2$ give
\begin{align*}
\mu'(t) 
&=\frac{1}{2\,t^2}\int_{0}^{\sqrt{t}}  \sqrt{\frac{u^2}{t-u^2}}  \, \D \lambda_{\chi}(u) \qquad \text{ with } \qquad \lambda_{\chi}(u)=u \chi''(1/u).
\end{align*}
Hence $G$ has a density $g$ with 
\begin{align*}
g(s)=\mu'(s^2)
=\frac{1}{2\,s^4}\int_{0}^{s}  \sqrt{\frac{1}{(s/u)^2-1}}  \, \D \lambda_{\chi}(u).
\end{align*}
Fubini's theorem and the substitution $s=1/t$ yield
\begin{align*}
&G(r)=\int_0^r g(s) \, \D s 
= \int_0^r \frac{1}{2\,s^4}\int_{0}^{s}  \sqrt{\frac{1}{(s/u)^2-1}}  \, \D \lambda_{\chi}(u) \, \D s\\
& = \frac{1}{2} \int_0^r \left( \int_{u}^{r} \frac{1}{\,s^4} {\frac{1}{\sqrt{s^2-u^2}}}  \, \D s\right) \, u \, \D \lambda_{\chi}(u)
= \frac{1}{2} \int_0^r \left( \int_{1/r}^{1/u}  {\frac{t^3}{\sqrt{1-t^2u^2}}}  \, \D t\right) \, u \, \D \lambda_{\chi}(u).
\end{align*}
Applying \cite[\page 96 2.264.4]{gradshteynryzhik_07} we arrive at
\begin{align*}
G(r)
&=\frac{1}{2} \int_0^r \left( \frac{1}{3u^2r^2} + \frac{2}{3u^4} \right) \left(\sqrt{1-\frac{u^2}{r^2}}\right) \, u \, \D \lambda_{\chi}(u) \\
&=\frac{1}{6\,r^3} \int_0^r \left( 1 + 2 \left(\frac{r}{u}\right)^2 \right)  \left( \sqrt{ \left( \frac{r}{u} \right)^2  -1} \right)  \, \D \lambda_{\chi}(u).
\end{align*}
To compute the shape function $f$ we apply ({\ref{eqn:fandH}}) with $G(s)=H(2s)$ 
\begin{align*}
\frac{\pi}{4} f\left(\frac{1}{2u}\right)
= \int_0^{u} s^2  g(s) \, \D s
= \frac{1}{2} \int_0^{u} \frac{1}{\,s^2}\int_{0}^{s}  \sqrt{\frac{1}{(s/t)^2-1}}  \, \D \lambda_{\chi}(t) \, \D s.
\end{align*}
By Fubini's theorem and the substitution $s=1/r$ we have
\begin{align*}
\frac{\pi}{4} f\left(\frac{1}{2u}\right)
&= \frac{1}{2} \int_0^{u} \left( \int_{t}^{u}  \frac{1}{\,s^2} {\frac{1}{\sqrt{s^2-t^2}}}  \, \D s \right) \, t \,  \D \lambda_{\chi}(t)\\
&= \frac{1}{2} \int_0^{u} \left( \int_{1/u}^{1/t}  {\frac{r}{\sqrt{1-r^2t^2}}}  \, \D r \right) \, t \,  \D \lambda_{\chi}(t).
\end{align*}
Applying \cite[\page 96 2.264.2]{gradshteynryzhik_07} gives
\begin{align*}
\frac{\pi}{4} f\left(\frac{1}{2u}\right)
= \frac{1}{2} \int_0^{u} \left(  \frac{\sqrt{u^2 - t^2}}{u \,t^2} \right) \, t \,  \D \lambda_{\chi}(t)
= \frac{1}{2\,u} \int_0^{u} \left( \sqrt{ \left(\frac{u}{t}\right)^2 - 1}  \right)  \,  \D \lambda_{\chi}(t).
\end{align*}
Finally, we replace $u$ by $1/(2u)$ and obtain
\begin{align*}
f(u) = \frac{4 \, u}{\pi} \int_0^{1/(2u)} \left( \sqrt{ \left(\frac{1}{2ut}\right)^2 - 1}  \right)  \,  \D \lambda_{\chi}(t)
\end{align*}
as desired.
\end{proof}
}

\begin{lemma}\label{lemma:erfcdensityfromLaplace}
Let $g(s)=\sqrt{\pi} \, f(s^2)$ be a probability density on $(0,\infty)$ and let \linebreak $\varphi:[0,\infty) \rightarrow [0,1]$ with $\varphi(0)=1$ be such that $-\varphi'\left(\sqrt{\cdot}\right)$ is the Laplace transform of $f$ in the following sense
\begin{align*}
-\varphi'\left(\sqrt{t}\right)=\int_0^\infty e^{-rt}f(r)\,\D r.
\end{align*} 
Then 
\begin{align*}
\varphi\left({t}\right)=\int_0^\infty \erfc\left({st}\right)g(s)\,\D s.
\end{align*} 
\end{lemma}

\begin{proof}(analogously to \cite[\page 104]{gneiting_99})
Replacing $t$ by $t^2$ and $r$ by $s^2$ yields
\begin{align*}
-\varphi'\left({t}\right)=\int_0^\infty 2s e^{-s^2t^2}f(s^2)\,\D s
= \int_0^\infty \frac{\D}{\D t} \left[-\erfc(st)\right] g(s)\,\D s.
\end{align*}
Applying Fubini's theorem when integrating \wrt $t$ gives 
\begin{align*}
\varphi(0)-\varphi(t)=\int_0^\infty \left[\erfc(0)-\erfc(st)\right] g(s) \D s,
\end{align*}
which entails the claim, since $g$ is a density on $(0,\infty)$ and $\varphi(0)=1$.
\end{proof}

\begin{proof}[Proof of Table~\ref{table:erfcmix}]
We apply Lemma~\ref{lemma:erfcdensityfromLaplace} and derive this table from known Laplace transforms in \cite{polyaninmanzhirov_08} using (in this order) equations [\page 964 5.3 (11)], [\page 964 5.3 (12), \page 963 5.2 (12) and \page 962 5.1 (26)], [\page 963 5.3 (1)] and [\page 963 5.3. (3) with $\nu=1.5$] therein.
\end{proof}

\noindent{\bf Acknowledgment}
Financial support for K. Strokorb by the German Research Foundation DFG through the Research Training Group 1023 and for M. Schlather by Volkswagen Stiftung within the 'WEX-MOP' project is gratefully acknowledged.

{\scriptsize
\bibliographystyle{hplain}
\bibliography{sbs_literature}
}


\end{document}